\documentclass[11pt]{article}
\usepackage{amsmath, graphicx, amsfonts,amssymb, calrsfs}
\usepackage{amsfonts}
\usepackage{mathrsfs}
\usepackage{amsthm}
\usepackage{color}

\def\bbR{\mathbb{R}}
\def\r{\rho}

\def\vrad{\mathrm{vrad}}

\def\cK{\mathcal{K}}

\def\E{\mathcal{E}}

\def\o{\omega}
\def\ball{B^n_2}
\def\polar{K^\circ}

\def\cS{\mathcal{S}}

\def\be{\begin{equation}}
\def\ee{\end{equation}}
\def\bea{\begin{eqnarray}}
\def\eea{\end{eqnarray}}
\def\bt{\begin{theorem}}
\def\et{\end{theorem}}
\def\bl{\begin{lemma}}
\def\el{\end{lemma}}
\def\br{\begin{remark}}
\def\er{\end{remark}}
\def\bc{\begin{corollary}}
\def\ec{\end{corollary}}
\def\bd{\begin{definition}}
\def\ed{\end{definition}}
\def\sphere{S^{n-1}}
\def\bp{\begin{proposition}}
\def\ep{\end{proposition}}
\def\bpf{\begin{proof}}
\def\epf{\end{proof}}
\def\dV{\,d \widetilde{V}_K(u)}

\def\affine{\widehat{\Omega}^{orlicz}_{\phi}}
\def\affineg{\widehat{G}^{orlicz}_{\phi}}
\def\affinep{\widehat{\Omega}^{orlicz}_{p}}
\def\affinegp{\widehat{G}^{orlicz}_{p}}
\newtheorem{theorem}{Theorem}[section]
\newtheorem{lemma}{Lemma}[section]
\newtheorem{remark}{Remark}[section]
\newtheorem{proposition}{Proposition}[section]
\newtheorem{corollary}{Corollary}[section]

\newtheorem{definition}{Definition}[section]
\begin{document}
\title{The Orlicz-Petty bodies
}

\author{Baocheng Zhu, Han Hong and Deping Ye }
\date{}
\maketitle
\begin{abstract}  This paper is dedicated to the Orlicz-Petty bodies. We first propose the homogeneous Orlicz affine and geominimal surface areas, and establish their basic properties such as homogeneity, affine invariance and affine isoperimetric inequalities. We also prove that the homogeneous geominimal surface areas are continuous, under certain conditions, on the set of convex bodies in terms of the Hausdorff distance. Our proofs rely on the existence of the Orlicz-Petty bodies and the uniform boundedness of the Orlicz-Petty bodies of a convergent sequence of convex bodies. Similar results for the nonhomogeneous Orlicz geominimal surface areas are proved as well.

 \end{abstract}\section{Introduction}

The theory of convex geometry was greatly enriched by the
combination of two notions: the volume and the linear Orlicz addition
of convex bodies \cite{Gardner2014, XJL}.  This new theory, usually called the
Orlicz-Brunn-Minkowski theory for convex bodies, started from the
works of Lutwak, Yang and Zhang \cite{LYZ2010a, LYZ2010b}, and
received considerable attention (see e.g., \cite{bor2013-1,
bor2012, bor2013-2, Chen2011, HaberlLYZ,  HabP, Ye2012, Ye2013, Zhu2012,
ZouXiong2014}). The linear Orlicz addition of convex
bodies was proposed by Gardner, Hug and Weil \cite{Gardner2014}
(independently Xi, Jin and Leng \cite{XJL}). Let $\varphi_i: [0,
\infty)\rightarrow [0, \infty)$, $i=1, 2$,  be  convex functions
such that $\varphi_i$ is strictly increasing with $\varphi_i(1)=1$,
$\varphi_i(0)=0$ and $\lim_{t\rightarrow \infty} \varphi_i(t)=\infty$. Let $S^{n-1}$ be
the unit sphere in $\bbR^n$ and $h_K: S^{n-1}\rightarrow (0,
\infty)$ denote the support function of $K$. For any given
$\varepsilon>0$ and two convex bodies $K$ and $L$ with the origin in
their interiors, the linear Orlicz addition $K+_{\varphi,
\varepsilon}L$ is determined by its support function
$h_{K+_{\varphi, \varepsilon}L}$, the unique solution of
\begin{equation*}
\varphi_1\Big(\frac{h_K(u)}{\lambda}\Big)+\varepsilon\varphi_2\Big(\frac{h_L(u)}{\lambda}\Big)=1\
\ \ \mbox{for} \ \ u\in S^{n-1}.\end{equation*}  Denote by
$|K+_{\varphi, \varepsilon}L|$ the volume of $K+_{\varphi,
\varepsilon}L$. If $(\varphi_1)'_l(1)$, the left derivative of $\varphi_1$ at $t=1$, exists and is positive, then  
\begin{equation}\label{geometricinterpretation-mixedvolume}\frac{(\varphi_1)'_l(1)}{n}
\cdot  \frac{\,d }{\,d\varepsilon}|K+_{\varphi, \varepsilon}L|
\bigg|_{\varepsilon=0^+}= \frac{1}{n} \int_{S^{n-1}}
\varphi_2\bigg(\frac{h_L(u)}{h_K(u)}\bigg)h_K(u)dS_K(u),\end{equation}
where  $S_K$ is the surface area measure of $K$ (see \cite{Gardner2014, XJL} for more details).  That is, formula
(\ref{geometricinterpretation-mixedvolume}) provides a geometric
interpretation of $V_{\phi}(K, L)$ for $\phi$ being convex and
strictly increasing. Here, for any continuous function $\phi: (0,
\infty)\rightarrow (0, \infty)$, $V_{\phi}(K, L)$ denotes the nonhomogeneous Orlicz
$L_{\phi}$ mixed volume of $K$ and $L$:
\begin{equation}\label{definition-mixed-volume-2016-07-11}
V_{\phi}(K, L)= \frac{1}{n} \int_{S^{n-1}}
\phi\bigg(\frac{h_L(u)}{h_K(u)}\bigg)h_K(u)dS_K(u).\end{equation}
To the best of our knowledge, there are no  geometric interpretations of
$V_{\phi}(K, L)$ for non-convex functions $\phi$ (even for $\phi(t)=t^p$ with $p<1$) in literature; and such geometric interpretations will be provided
in Subsection \ref{section-geom-inter} in this paper.  Note that
formula (\ref{geometricinterpretation-mixedvolume}) is essential for
the Orlicz-Minkowski inequality and many other objects, such as the
Orlicz affine and geominimal surface areas \cite{Ye2015b}.

Introduced by Blaschke in 1923 \cite{Bl1}, the classical affine
surface area was thought
 to be one of the core concepts in
the Brunn-Minkowski theory of convex bodies due to its important
applications in, such as, approximation of convex bodies by
polytopes \cite{Gr2, LSW, SW5} and valuation theory \cite{A1, A2,
LudR}. Since the groundbreaking paper by Lutwak \cite{Lu1},
considerable progress has been made on the theory of the $L_p$ affine
surface areas (see e.g., \cite{Jenkinson2012, LR1,  MW1, MW2,
Paouris2010, SW4, Werner2012a, Werner2012b, WY2008}). Like the
classical affine surface area, the $L_p$ affine surface areas play
fundamental roles in applications and provide powerful tools in
convex geometry. Note that the $L_p$ affine surface areas are affine
invariant valuations with homogeneity.

In the Orlicz-Brunn-Minkowski theory for convex bodies, a central
task is to find the ``right" definitions for the Orlicz affine
surface areas. Here, we will discuss two different approaches by
Ludwig  \cite{Ludwig2009} and the third author \cite{Ye2015b}.  Based on an integral formula, Ludwig proposed the
general affine surface areas \cite{Ludwig2009}. Ludwig's definitions
work perfectly in studying properties such as valuation
\cite{Ludwig2009}, the characterization of valuation \cite{HabP,
Ludwig2009} and the monotonicity under the Steiner symmetrization
\cite{Ye2013}. In order to define the Orlicz geominimal surface
areas, new ideas are needed because geominimal surface areas do not
have convenient integral expression like their affine relatives. The
third author provided a unified approach to define the Orlicz affine
and geominimal surface areas  \cite{Ye2015b} based on the Orlicz
$L_{\phi}$ mixed volume $V_{\phi}(\cdot, \cdot)$ defined in formula
(\ref{definition-mixed-volume-2016-07-11}). In fact, the approach in
\cite{Ye2015b} is related to an optimization problem for the
$f$-divergence \cite{HouYe2016} and could be used to define other
versions of Orlicz affine and geominimal surface areas
\cite{CaglarYe2016, Ye2016a, YeZhuZhou2015}.

Note that the natural property of ``homogeneity"  is missing in the Orlicz 
affine surface areas in \cite{Ludwig2009, Ye2015b}.  To define the
Orlicz affine surface areas with homogeneity is one of the main
objects in this paper; and it will be done in Section
\ref{section-homogeneous-geom}.  As an example, we give the
definition for $\phi\in \widehat{\Phi}_1$, where $\widehat{\Phi}_1$
is the set of functions $\phi: [0, \infty)\rightarrow [0, \infty)$
such that $\phi$  is strictly increasing with $\phi(0)=0,
\phi(1)=1$, $\lim_{t\rightarrow \infty} \phi(t)=\infty$ and
$\phi(t^{-1/n})$ being strictly convex on $(0, \infty)$. For convex body $K$ and star
body $L$ with the origin in their interiors, define $\widehat{V}_{\phi} (K, L^\circ)$  for $\phi\in
\widehat{\Phi}_1$ by  \begin{equation*}  \widehat{V}_{\phi} (K,
L^\circ)=\inf_{\lambda>0}  \bigg\{ \int _{S^{n-1}}
\phi\Big(\frac{n|K|  }{\lambda \cdot \rho_L(u)\cdot
h_K(u)}\Big)h_K(u)\, dS_K(u)\leq n|K|\bigg\},\end{equation*} where
$\rho_L$ denotes the radial function of $L$.  We
now define $\affine(K)$ for $\phi\in \widehat{\Phi}_1$, the
homogeneous Orlicz $L_{\phi}$ affine surface area of $K$, by the
infimum of $\widehat{V}_{\phi} (K, L^\circ)$ where $L$  runs over
all star bodies with the origin in their interiors and
$|L|=|\ball|$ (the volume of the Euclidean unit ball of $\bbR^n$). In
Proposition \ref{affine-invariance:proposition}, we show that
$\affine(K)$ is invariant under the volume preserving linear maps
and has homogeneous degree $(n-1)$.  Moreover, the following affine
isoperimetric inequality is established in Theorem \ref{homogeneous
Orlicz affine isoperimetric inequality}:  {\em if $K$ has its centroid
at the origin and $\phi\in \widehat{\Phi}_1$, then
\begin{eqnarray*}\frac{\affine(K)}{\affine(\ball)}\ \leq \
\bigg(\frac{|K|}{|\ball|}\bigg)^{\frac{n-1}{n}},  \end{eqnarray*}
with equality if and only if $K$ is an origin-symmetric ellipsoid.}
Note that affine isoperimetric inequalities are fundamental in
convex geometry; and these inequalities compare affine invariant
functionals with the volume  (see e.g., \cite{CG, HaberlFranz2009,
LutwakZhang1997, LYZ1, LYZ2010a, LYZ2010b, WY2008, Zhang2007}).

 The Petty body and its $L_p$ extensions for $p>1$ were used to study the continuity of the classical geominimal surface area and its $L_p$ counterparts \cite{Lu1, Petty1974}. To prove the existence and uniqueness of the Orlicz-Petty bodies is one of the main goals of Section \ref{section-homogeneous-petty} in  this paper. In order to fulfill these goals, we first define $\affineg(K)$, the homogeneous Orlicz geominimal surface area of $K$,  by the infimum of $\widehat{V}_{\phi} (K, L^\circ)$ where $L$  runs over all convex bodies with the origin in their interiors and $|L|=|\ball|$. The classical geominimal surface area, which corresponds to $\phi(t)=t$, was introduced by Petty \cite{Petty1974} in order to study the affine isoperimetric problems \cite{Petty1974, Petty1985}. The classical geominimal surface area and its $L_p$ extensions  (corresponding to $\phi(t)=t^p$) for $p>1$ by Lutwak \cite{Lu1} are continuous on the set of convex bodies in terms of the Hausdorff distance; while their affine relatives are only semicontinuous. The main ingredients to prove the continuity of the $L_p$ geominimal surface area for $p\geq 1$ are the existence of the $L_p$ Petty bodies and the uniform boundedness of the $L_p$ Petty bodies of a convergent sequence of convex bodies (hence, the Blaschke selection theorem can be used).  In Section \ref{section-homogeneous-petty}, we will prove that $\affineg(\cdot)$ is also continuous for $\phi\in \widehat{\Phi}_1$. Note that  $\phi(t)=t^p\in \widehat{\Phi}_1$ if $p\in (0, \infty)$. Consequently, the $L_p$ geominimal surface area for $p\in (0, 1)$, proposed by the third author in \cite{Ye2015a}, is also continuous.  Our approach basically follows the steps in  \cite{Lu1, Petty1974}; however, our proof is more delicate and requires much more careful analysis due to the lack of convexity of $\phi$ (note that in $\widehat{\Phi}_1$, $\phi(t^{-1/n})$ is assumed to be convex, not $\phi$ itself). In particular, we prove the existence and uniqueness of the Orlicz-Petty bodies in Proposition \ref{p3}. Our main result is Theorem \ref{continuity-0717-1}: {\em if $\phi\in\widehat{\Phi}_1$, then the homogeneous $L_\phi$ Orlicz geominimal surface area is continuous  on the set of convex bodies with respect to the Hausdorff distance.} The continuity of nonhomogeneous Orlicz geominimal surface areas \cite{Ye2015b} will be discussed in Subsection \ref{section-nonhomogeneous-petty}. The $L_p$ Petty body for $p\in (-1, 0)$ is more involved and will be discussed in Section \ref{section-symmetric-geom}.


\section{Background and Notation}\label{section 2}

We now introduce the basic well-known facts and standard notations needed in this paper. For more details and more concepts in convex geometry, please see \cite{Gard, Gruber2007,Sch}.

A convex and compact subset $K\subset\bbR^n$ with nonempty interior is called a convex body in $\bbR^n$. By $\cK$ we mean the set of all convex bodies containing the origin and by $\cK_0$ the set of all convex bodies with the origin in their interiors. A convex body $K$ is said to be origin-symmetric if $K=-K$ where $-K=\{x\in \bbR^n: -x\in K\}$. Let $\cK_e$ denote the set of all origin-symmetric convex bodies in $\bbR^n$.  The volume of $K$ is denoted by $|K|$ and the volume radius of $K$ is denoted by $\vrad(K)$. By $\ball$ and $\sphere$, we mean the Euclidean unit ball and the unit sphere in $\bbR^n$ respectively. The volume of $\ball$ will be often written by $\omega_n$ and the natural spherical measure on $\sphere$ is written by $\sigma$.  Consequently, $\vrad(K)=(|K|/\omega_n)^{1/n}$. The standard notation $GL(n)$ stands for the set of all invertible linear transforms on
$\bbR^n$. For $A\in GL(n)$, we use $\det A$ to denote the determinant of
$A$. Let $SL(n)=\{A: \ A\in GL(n)\  \mbox{and} \  \det A=\pm 1\} $. By $A^t$ and $A^{-t}$  we mean the transpose of $A$ and the inverse of $A^t$ respectively.

Each convex body $K\in\mathcal{K}$ has a continuous support function $h_K: \sphere\rightarrow [0, \infty)$ defined by
$h_K(u)=\max_{x\in K} \langle x, u\rangle$ for $u\in \sphere$,  where $\langle \cdot,
\cdot \rangle$ denotes the usual inner product. Note that $h_K$ for $K\in \cK$ is nonnegative on $\sphere$, but it is strictly positive on $\sphere$ if $K\in \cK_0$.  Moreover, one can define a probability measure $\widetilde{V}_K$ on each $K\in \cK$  by $$
\dV=\frac{h_K(u)\,dS_K(u)}{n|K| } \ \ \ \ \mbox{for} \ \ u\in \sphere, $$
where $S_K$ is the surface area measure of $K$.  It is well known that $S_K$ satisfies \begin{eqnarray} \int_{\sphere}u\,d S_K(u)=0 \ \ \mbox{and}\ \ \int_{\sphere} |\langle u, v\rangle|\,dS_K(u)>0 \ \mbox{for each}\ v\in \sphere. \label{minkowski-solution-1} \end{eqnarray} The first formula of (\ref{minkowski-solution-1}) asserts that $S_K$ has its centroid at the origin and the second one states that $S_K$  is not concentrated on any great subsphere. Let $\mu_K$ denote the usual surface area of $\partial K$, the boundary of $K$, and $N_K(x)$ denote a unit outer normal vector of $x\in \partial K$. For each $f\in C(S^{n-1})$, where $C(S^{n-1})$ denotes the set of all continuous functions defined on $\sphere$, one has
$$\int_{\sphere} f(u)\,dS_K(u)=\int _{\partial K}
f(N_K(x))\,d\mu_K(x).$$ The dilation of $K$ is of form $sK=\{sx: x\in K\}$ for $s>0$. Clearly, $h_{sK}(u)=s\cdot h_K(u)$ for all $u\in \sphere$. Moreover, $sK$ and $K$ share the same  probability measure $\dV$. Two convex bodies $K$ and $L$ are said to be dilates of each other if $K=s L$ for some constant $s>0$.

For $u\in \sphere$, let  $l_u=\{tu: t\geq 0\}$ . We say $L\subset \bbR^n$ is star-shaped at the origin if, for each $u\in S^{n-1}$, $L\cap l_u$ is a closed line segment containing the origin. One can define the radial function $\rho_L: \sphere\rightarrow [0, \infty)$ for $L$ a star-shaped set about
the origin by $$\rho_L(u)=\max\{\lambda\geq0:\lambda u\in L\} \  \ \ \ \ \mbox{for} \ \  u\in \sphere. $$ If $\rho_L$ is
positive and continuous on $\sphere$, then $L$ is called a star body about the
origin. Denote by $\cS_0$ the set of star bodies about the origin in $\mathbb{R}^n$ and clearly $\cK_0\subset\cS_0$. The volume of $L\in \cS_0$ can be calculated by \be\label{formula for volume}
  |L|=\frac{1}{n} \int_{S^{n-1}}\rho_L^n(u)\, d\sigma(u)\ \ \ \mathrm{and} \ \ \  |\polar|=\frac{1}{n} \int_{S^{n-1}}\frac{1}{h_K^n(u)}\,d\sigma(u) .
 \ee Hereafter, $\polar\in \cK_0$ is the polar body of $K\in \cK_0$; and the support function $h_{\polar}$ and the radial function $\rho_{\polar}$ are given by $$
h_{K^{\circ}}(u)=\frac{1}{\rho_{K}(u)} \quad\mathrm{and} \quad
\rho_{K^{\circ}}(u)=\frac{1}{h_{K}(u)}, \quad \mathrm{for\ all}\ \  u\in
S^{n-1}.
$$ Alternatively, $\polar$ can be defined by $$K^\circ=
\{x\in\mathbb{R}^{n}: \langle x, y\rangle \leq 1\ \ \mathrm{for\ all}\ y\in
K \}.$$  The bipolar theorem states that $(\polar)^\circ=K$ if $K\in \cK_0$.

Let $\cK_c\subset \cK_0$ be the set of convex bodies with their centroids at origin; that is,  $\int_{K} x\,dx=0$ if $K\in \cK_c$.  We say $K\in \cK_0$ has the Santal\'{o} point at the origin if $\polar\in \cK_c$. Denote by $\cK_s\subset \cK_0$ the set of convex bodies with their Santal\'{o} points at the origin, and let $\widetilde{\cK}= \cK_s\cup \cK_c$. The set $\widetilde{\cK}$ is important in the famous Blaschke-Santal\'{o} inequality: for $K\in \widetilde{\cK}$, one has $$|K|\cdot |\polar|\leq \omega_n^2$$ with equality if and only if $K$ is an origin-symmetric ellipsoid (i.e., $K=A(\ball)$ for some $A\in GL(n)$).

On the set $\cK$, we consider the topology generated by the Hausdorff distance  $d_H(\cdot, \cdot)$. For $K, K'\in \cK$, define $d_H(K, K')$ by $$d_H(K, K')=\|h_K-h_{K'}\|_{\infty} =\sup_{u\in \sphere} |h_K(u)-h_{K'}(u)|.$$ A sequence $\{K_i\}_{i\geq 1}\subset \cK$ is said to be convergent to a convex body $K_0$ if $d_H(K_i, K_0)\rightarrow 0$ as $i\rightarrow \infty$. Note that if $K_i\rightarrow K_0$ in the Hausdorff distance, then $S_{K_i}$ is weakly convergent to $S_{K_0}$. That is,
for all $f\in C(S^{n-1})$, one has $$ \lim_{i\rightarrow\infty} \int_{\sphere} f(u)\,dS_{K_i}(u) = \int_{\sphere} f(u)\,dS_{K_0}(u).$$ We will  use a modified form of the above limit: if  $\{f_i\}_{i\geq 1}\subset C(S^{n-1})$ is uniformly convergent to $f_0\in C(S^{n-1})$ and $\{K_i\}_{i\geq 1}\subset \cK$ converges to  $K_0\in \cK$ in the Hausdorff distance, then   \be\label{weak fact}\lim_{i\rightarrow\infty}\int_{S^{n-1}}f_i(u)\
d S_{K_i}(u)=\int_{S^{n-1}}f_0(u)\ dS_{K_0}(u).  \ee  The Blaschke selection
theorem is a powerful tool in convex geometry (see e.g.,
\cite{Gruber2007, Sch}) and will be often used in this paper. It reads: {\em  every bounded sequence of convex bodies has a subsequence
that converges to a convex body.}

The following result, proved by Lutwak \cite{Lu1}, is essential for our main results.

\bl  \label{l1} Let $\{K_i\}_{i\geq 1}\subset \mathcal{K}_{0}$ be a convergent sequence with limit $K_0$, i.e., $K_i\rightarrow K_0$ in the Hausdorff distance. If the sequence
$\{|K_i^\circ|\}_{i\geq 1}$ is bounded, then $K_0\in\mathcal{K}_{0}$. \el

 \section{The homogeneous Orlicz affine and geominimal surface areas}\label{section-homogeneous-geom}

This section is dedicated to  Orlicz  affine and geominimal  surface areas with homogeneity. Let $\mathcal{I}$ denote the set of continuous functions $\phi: [0, \infty)\rightarrow [0, \infty)$ which are strictly increasing with $\phi(1)=1$, $\phi(0)=0$ and $\phi(\infty)=\lim_{t\rightarrow \infty} \phi(t)=\infty$. Similarly, $\mathcal{D}$ denotes  the set of continuous functions $\phi: (0, \infty)\rightarrow (0, \infty)$ which are strictly decreasing with $\phi(1)=1$,  $\phi(0)=\lim_{t\rightarrow 0} \phi(t)=\infty$ and $\phi(\infty)=\lim_{t\rightarrow \infty} \phi(t)=0$. Note that the conditions on $\phi(0), \phi(1)$  and $\phi(\infty)$ are  mainly for convenience; results may still hold for more general strictly increasing or decreasing functions.

 The Orlicz $L_{\phi}$ mixed volume of convex bodies $K$ and $L$, $V_{\phi}(K, L)$, given in formula (\ref{definition-mixed-volume-2016-07-11}) does not have homogeneity  in general. In order to define the homogeneous Orlicz affine and geominimal surface areas, a homogeneous Orlicz $L_{\phi}$ mixed volume of convex bodies $K$ and $L$, denoted by  $\widehat{V}_{\phi} (K, L)$, is needed.

 \bd \label{homogeneous-mixed-volume-11}  For $K, L\in \cK_0$ and $\phi\in \mathcal{I}$,  define $\widehat{V}_{\phi} (K, L)$ by  \begin{equation} \label{mixed volume hat} \widehat{V}_{\phi} (K, L)=\inf_{\lambda>0}  \bigg\{ \int _{S^{n-1}} \phi\Big(\frac{n|K|\cdot h_L(u)}{\lambda \cdot  h_K(u)}\Big)\dV\leq 1\bigg\}.\end{equation} While if $\phi\in \mathcal{D}$, $\widehat{V}_{\phi} (K, L)$  is defined as above with `` $\leq 1$" replaced by `` $\geq 1$".  \ed

Clearly  $\widehat{V}_{\phi} (K, L)>0$ for $K, L\in \cK_0$. Definition \ref{homogeneous-mixed-volume-11} is motivated by formula (10.5) in \cite{Gardner2014} with a slight modification; namely, an extra term $n|K|$ has been added in the numerator of the variable inside $\phi$. This extra term $n|K|$ is added in order to get, as $\phi(1)=1$,   \begin{equation} \widehat{V}_{\phi} (K, K)=n|K|. \label{equal to volume} \end{equation}  Formula (\ref{mixed volume hat}) coincides with formula (10.5) in \cite{Gardner2014} if $\phi\in \mathcal{I}$ is convex.

 The following corollary states the homogeneity of $\widehat{V}_{\phi} (K, L)$, which has been made to be the same as the classical mixed volume $V_1(K, L)$.

 \bc \label{corollary:homogeneous-1} Let $s, t>0$ be constants. For $K, L\in \cK_0$, one has, for $\phi\in \mathcal{I}\cup \mathcal{D}$,
 \begin{equation} \label{mixed volume-homogeneous-1}   \widehat{V}_{\phi} (sK, tL)=s^{n-1}t\cdot  \widehat{V}_{\phi} (K, L).  \end{equation}   \ec

 \begin{proof} For $\phi\in \mathcal{I}$, one has, by letting $\eta=s^{n-1}t\lambda$,  \begin{eqnarray*} \widehat{V}_{\phi} (sK, tL) &=& \inf_{\eta>0}  \bigg\{ \int _{S^{n-1}} \phi\Big(\frac{t\cdot n|K|\cdot  h_L(u)}{\eta \cdot  s^{1-n} \cdot h_K(u)}\Big)\dV\leq 1\bigg\} \\ &=& s^{n-1}t\cdot \inf_{\lambda>0}  \bigg\{ \int _{S^{n-1}} \phi\Big(\frac{n|K|\cdot h_L(u)}{\lambda \cdot  h_K(u)}\Big)\dV\leq 1 \bigg\}.\end{eqnarray*}  That is, $\widehat{V}_{\phi} (sK, tL)=s^{n-1}t\cdot  \widehat{V}_{\phi} (K, L). $  In particular, if $s=1$ and $t>0$, then $$ \widehat{V}_{\phi} (K, tL)=t\cdot \widehat{V}_{\phi} (K, L);$$  while if $t=1$ and $s>0$, then  $$ \widehat{V}_{\phi} (sK, L)=s^{n-1}\cdot \widehat{V}_{\phi} (K, L).$$
  The case for $\phi\in\mathcal{D}$ follows along the same way.\end{proof}

Let the function $G: (0, \infty)\rightarrow (0, \infty)$ be given by $$G(\lambda)=  \int _{S^{n-1}} \phi\Big(\frac{n|K|\cdot h_L(u)}{\lambda \cdot  h_K(u)}\Big)\dV.$$ For $\phi\in \mathcal{I}$,  the function $G$ is strictly decreasing on $\lambda$ with $$\lim_{\lambda\rightarrow 0} G(\lambda)=\lim_{t \rightarrow \infty} \phi(t) \ \ \ \mbox{and} \ \ \  \lim_{\lambda \rightarrow \infty} G(\lambda)=\lim_{t \rightarrow 0} \phi(t). $$ As an example, we show that $\lim_{\lambda\rightarrow 0} G(\lambda)=\lim_{t \rightarrow \infty} \phi(t)$. To this end,  as $\phi\in \mathcal{I}$ is strictly increasing,  we have  \begin{eqnarray*} G(\lambda) &=&  \int _{S^{n-1}} \phi\Big(\frac{n|K|\cdot h_L(u)}{\lambda \cdot  h_K(u)}\Big)\dV \\  &\geq& \int _{S^{n-1}} \phi\bigg(\frac{n|K|\cdot \min_{u\in \sphere} h_L(u)}{\lambda \cdot  \max_{u\in \sphere} h_K(u)}\bigg)\dV  \\ &=& \phi\bigg(\frac{n|K|\cdot \min_{u\in \sphere} h_L(u)}{\lambda \cdot  \max_{u\in \sphere} h_K(u)}\bigg). \end{eqnarray*}  This yields  $$ \lim_{\lambda\rightarrow 0} G(\lambda)  \geq   \lim_{\lambda\rightarrow 0}\phi\bigg(\frac{n|K|\cdot \min_{u\in \sphere} h_L(u)}{\lambda \cdot  \max_{u\in \sphere} h_K(u)}\bigg)=\lim_{t \rightarrow \infty} \phi(t).$$ Similarly, one has  $$ \lim_{\lambda\rightarrow 0} G(\lambda)  \leq   \lim_{\lambda\rightarrow 0}\phi\bigg(\frac{n|K|\cdot \max_{u\in \sphere} h_L(u)}{\lambda \cdot  \min_{u\in \sphere} h_K(u)}\bigg)=\lim_{t \rightarrow \infty} \phi(t),$$  and the desired result follows.  On the other hand,  the function $G$ for $\phi\in \mathcal{D}$ is strictly increasing on $\lambda$ with  $$\lim_{\lambda\rightarrow 0} G(\lambda)=\lim_{t \rightarrow \infty} \phi(t) \ \ \ \mbox{and} \ \ \  \lim_{\lambda \rightarrow \infty} G(\lambda)=\lim_{t \rightarrow 0} \phi(t). $$  Together with $\phi(1)=1$,  we have proved the following corollary.

  \bc \label{corollary:homogeneous-2} Let $\phi\in \mathcal{I}\cup \mathcal{D}$ and $K, L\in \cK_0$. Then $\widehat{V}_{\phi} (K, L)>0$, and $\lambda_0= \widehat{V}_{\phi} (K, L)$  if and only if $$G(\lambda_0)= \int _{S^{n-1}} \phi\Big(\frac{n|K|\cdot h_L(u)}{\lambda_0 \cdot  h_K(u)}\Big)\dV=1.$$ \ec    For $\phi(t)=t^p$, one writes $\widehat{V}_{p} (K, L)$ instead of $\widehat{V}_{\phi} (K, L)$. A simple calculation shows that \begin{eqnarray*} \widehat{V}_{p} (K, L) &=&n|K| \cdot \bigg[\int _{S^{n-1}} \bigg(\frac{h_L(u)}{h_K(u)}\bigg)^p \dV\bigg]^{1/p}\\ &=& (n|K|)^{1-\frac{1}{p}}\cdot \big(n {V}_{p} (K, L)\big)^{1/p},\end{eqnarray*} where $V_p(K, L)$ is the $L_p$ mixed volume of $K$ and $L$ for $p\in \bbR$ \cite{Lu1, Ye2015a}, i.e.,  $${V}_{p} (K, L)= \frac{1}{n}\int _{S^{n-1}} h_L(u)^p  h_K(u)^{1-p}\, dS_K(u). $$

If $\phi\in \mathcal{I}$ is convex, the Orlicz-Minkowski inequality holds \cite{Gardner2014}: for $K, L\in\cK_0$, one has
\begin{equation}\label{Minkowski-inequality-2016-713} \widehat{V}_\phi(K,L)\geq n\cdot {|K|}^\frac{n-1}{n}
{|L|}^\frac{1}{n}. \end{equation}
If in addition $\phi$ is strictly convex, equality holds if and only if
$K$ and $L$ are dilates to each other. In particular, the classical Minkowski inequality is related to $\phi(t)=t$: for $K,
L\in \cK_0$, one has \be\label{classical Minkowski inequality} V_1(K,
L)^{n}\geq |K|^{n-1}|L|, \ee with equality  if and only if $K$ and
$L$ are homothetic to each other (i.e., there exist a constant $s>0$ and a vector $a\in \bbR^n$ such that $K=sL+a$).

In order to define the homogeneous Orlicz affine surface areas, we need to define $\widehat{V}_{\phi} (K, L^\circ)$ for  $L\in \cS_0$. The definition is similar to Definition \ref{homogeneous-mixed-volume-11} but with $h_{L^\circ}$ replaced by $1/\rho_L$. That is,  for $\phi\in \mathcal{I}\cup \mathcal{D}$, $\widehat{V}_{\phi} (K, L^\circ)$ for $K\in \cK_0$ and $L\in \cS_0$  is defined by the constant $\lambda_0$ such that  \begin{equation} \label{mixed volume hat-star} \int
_{S^{n-1}} \phi\Big(\frac{n|K|}{\lambda_0 \cdot \rho_L(u)\cdot
h_K(u)}\Big)\dV =1.\end{equation} Of course, $\widehat{V}_{\phi} (K, L^\circ)$ for $K, L\in \cK_0$ given by formula
(\ref{mixed volume hat-star}) coincides with the one given by
formula (\ref{mixed volume hat}). Note that $\widehat{V}_{\phi} (K,
L^\circ)$ for $K\in \cK_0$ and $L\in \cS_0$ is also homogeneous as stated in Corollary \ref{corollary:homogeneous-1}.

For function $\phi \in \mathcal{I}\cup\mathcal{D}$, let $F(t)=\phi(t^{-1/n})$ and hence $\phi(t)=F(t^{-n})$. The
relations between $\phi$ and $F$ have been  discussed in \cite{Ye2015b}. For example, a): $\phi$ and $F$ have opposite monotonicity, that is, if one is strictly decreasing (increasing), then the other one will be strictly increasing (decreasing); b): if one is
convex and increasing, then the other one is convex and decreasing. As mentioned in \cite{Ye2015b}, to define Orlicz affine and geominimal surface areas, one needs to consider the convexity and concavity of $F$ instead of the convexity and concavity of $\phi$ itself.  Let
\begin{eqnarray*}   \widehat{\Phi}_1&=&\big\{\phi: \phi\in \mathcal{I}  \ \mbox{and} \ F \ \mbox{is strictly convex} \big\};\\
\widehat{\Phi}_2 &=&\{\phi: \phi\in
\mathcal{D} \ \mbox{and} \ F \ \mbox{is strictly concave}\}.\end{eqnarray*} We often use $\widehat{\Phi}$ for $\widehat{\Phi}_1\cup\widehat{\Phi}_2$.  Sample functions in  $\widehat{\Phi}$ are: $t^p$ with $p\in (-n,
0)\cup (0, \infty)$.  Similarly,   let
\begin{eqnarray*} \widehat{\Psi}=\{\phi: \phi\in \mathcal{D}  \ \mbox{and} \ F \ \mbox{is strictly convex}\}.\end{eqnarray*} Note that if $\phi\in\mathcal{I}$
such that $F$ is strictly concave, then $\phi$ is a constant. We are
not interested in this case.   The set $\widehat{\Psi}$ contains
functions such as $t^p$ with $p\in (-\infty, -n)$.

\bd \label{definition:homogeneous:Orlicz:affine:surface}  Let $K\in \cK_0$.  The homogeneous Orlicz $L_{\phi}$ affine surface area of $K$, denoted by $\affine(K)$, is defined  by \begin{eqnarray}\label{affine surface area hat}
 \affine(K)&=&\inf_{L\in \cS_0} \Big\{\widehat{V}_{\phi} (K, \vrad(L)L^\circ)\Big\} \ \ \ \mathrm{for} \ \ \phi\in \widehat{\Phi};  \\
 \label{geominimal surface area hat-0712}
 \affine(K)&=&\sup_{L\in \cS_0} \Big\{\widehat{V}_{\phi} (K, \vrad(L)L^\circ)\Big\} \ \ \ \mathrm{for} \ \ \phi\in \widehat{\Psi}. \end{eqnarray}
  The homogeneous Orlicz $L_{\phi}$ geominimal surface area of $K$,  denoted by $\affineg(K)$, is defined similarly with $\cS_0$ replaced by  $\cK_0$.  \ed

Clearly $\affine(K)\leq \affineg(K)$ if $\phi\in \widehat{\Phi}$ and $\affine(K)\geq \affineg(K)$ if $\phi\in \widehat{\Psi}$.
 For $\phi(t)=t^p$, one writes $\affinep(K)$ instead of $\affine(K)$. In particular, for $-n\neq p\in \bbR$,   $$\affinep(K)= (n\o_n)^{-1/n}\cdot\big(as_p(K)\big)^{\frac{n+p}{np}}\cdot (n|K|)^{1-\frac{1}{p}},$$ where $as_p(K)$ is the $L_p$ affine surface area of $K$ (see e.g., \cite{Lu1, Ye2015a}):  \begin{eqnarray*}\label{Lp-affine-surface-area-Lutwak--1}
as_p(K)&=&\inf_{L\in \cS_0} \left\{n V_p(K, L^\circ) ^{\frac{n}{n+p}}\ \nonumber
|L|^{\frac{p}{n+p}}\right\}, \ \ \ p\geq 0; \\\label{Lp-affine-surface-area-Lutwak--2}
as_p(K)&=&\sup_{L\in \cS_0} \left\{n V_p(K, L^\circ) ^{\frac{n}{n+p}}\
|L|^{\frac{p}{n+p}}\right\}, \ \ \ -n\neq p<0.
\end{eqnarray*}  Similarly, for $-n\neq p\in \bbR$,  $$\affinegp(K)= (n\o_n)^{-1/n}\cdot\big(\tilde{G}_p(K)\big)^{\frac{n+p}{np}}\cdot (n|K|)^{1-\frac{1}{p}},$$  where $\tilde{G}_p(K)$ is the $L_p$ geominimal surface area \cite{Lu1, Ye2015a}:  \begin{eqnarray*}\label{Lp-geominimal-surface-area-ye--1}
\tilde{G}_p(K)&=&\inf_{L\in \cK_0} \left\{n V_p(K, L^\circ) ^{\frac{n}{n+p}}\
|L|^{\frac{p}{n+p}}\right\}, \ \ \ p\geq 0; \\\label{Lp-geominimal-surface-area-ye--2}
\tilde{G}_p(K)&=&\sup_{L\in \cK_0} \left\{n V_p(K, L^\circ) ^{\frac{n}{n+p}}\
|L|^{\frac{p}{n+p}}\right\}, \ \ \ -n\neq p<0.
\end{eqnarray*}

When $K=\ball$, both $\affine(\ball)$ and $\affineg(\ball)$ can be calculated precisely.

\bc \label{equal-ball-0715} For $\phi\in \widehat{\Phi}\cup \widehat{\Psi}$, one has
\begin{equation} \label{ball-2} \affine(\ball)=\affineg(\ball)=n\omega_n. \end{equation}  \ec

\begin{proof}  We only prove  $\affine(\ball)=n\omega_n$ with $\phi\in \widehat{\Phi}_1$, and the other cases follow along the same lines. As $\phi\in \widehat{\Phi}_1$, one sees that $\phi$ is strictly increasing and $F(t)=\phi(t^{-1/n})$ is strictly convex.
First of all, by formulas (\ref{equal to volume}) and (\ref{affine surface area hat}), one has \begin{eqnarray}\label{ball-1} \affine(\ball) \leq  \widehat{V}_{\phi}(\ball ,  \ball) = n\omega_n.
  \end{eqnarray} From Corollary \ref{corollary:homogeneous-2} and Jensen's inequality, the fact that $F$ is strictly convex yields \begin{eqnarray*} 1&=&\int _{S^{n-1}} \phi\bigg(\frac{n\omega_n \cdot \vrad(L)}{\widehat{V}_{\phi}(\ball, \vrad(L) L^\circ) \cdot  \rho_L(u)}\bigg)\cdot \frac{1}{n\omega_n}\,d\sigma(u)   \\ &\geq& F\bigg( \int _{S^{n-1}} \frac{[\widehat{V}_{\phi}(\ball, \vrad(L) L^\circ)]^n \cdot  \rho_L^n(u)} {[n\omega_n \cdot \vrad(L)]^n\cdot n\omega_n}\,d\sigma(u)\bigg)  \\ &=& \phi\bigg(\frac{n\omega_n}{\widehat{V}_{\phi}(\ball, \vrad(L) L^\circ)}\bigg).  \end{eqnarray*} As $\phi(1)=1$ and $\phi$ is strictly  increasing, one gets, for all $L\in \cS_0$,  $${\widehat{V}_{\phi}(\ball, \vrad(L) L^\circ)}\geq {n\omega_n}. $$ The desired equality follows by taking the infimum over $L\in \cS_0$ and by formula (\ref{ball-1}). \end{proof}

  \begin{proposition}\label{affine-invariance:proposition} Let $K\in \cK_0$ and $A\in GL(n)$. For $\phi\in \widehat{\Phi}\cup\widehat{\Psi}$, one has  $$\affine(AK) =|\det A|^{\frac{n-1}{n}} \cdot  \affine(K) \ \ \ \mathrm{and} \ \ \ \affineg(AK) =|\det A|^{\frac{n-1}{n}} \cdot  \affineg(K).$$  \end{proposition}

\begin{proof}   Let $A\in GL(n)$ and $\|\cdot\|$ be the usual Euclidean norm. For $v\in S^{n-1}$, let $u=u(v)=\frac{A^{t} v}{\|A^{t} v\|}$. By the definitions of support and radial functions, one can easily check that  $$h_{AK}(v)=\|A^t v\| \cdot h_K(u) \ \ \ \mathrm{and} \ \ \  \r_{A^{-t} L}(v){\|A^tv\|}= {\r_L(u)}.$$      Consequently, $h_{AK}(v) \r_{A^{-t} L}(v) = h_K(u)   {\r_L(u)}$ and   \begin{eqnarray*}   \widehat{V}_{\phi} (AK, (A^{-t} L)^\circ) &=&\inf_{\lambda>0}  \bigg\{ \int _{S^{n-1}} \phi\Big(\frac{n|AK|}{\lambda \cdot h_{AK}(v) \r_{A^{-t} L}(v)}\Big)\,d\widetilde{V}_{AK}(v)\leq 1\bigg\}\\&=&\inf_{\lambda>0}  \bigg\{ \int _{S^{n-1}} \phi\Big(\frac{|\det A| \cdot n|K|}{\lambda \cdot h_K(u)   {\r_L(u)}}\Big)\dV\leq 1\bigg\}\\ &=&|\det A| \cdot \inf_{\eta>0}  \bigg\{ \int _{S^{n-1}} \phi\Big(\frac{ n|K|}{\eta \cdot h_K(u)   {\r_L(u)}}\Big)\dV\leq 1\bigg\},
 \end{eqnarray*} where $\lambda=|\det A| \cdot \eta$.  Consequently,
 \be\label{transform}
 \widehat{V}_{\phi} (AK, (A^{-t} L)^\circ)=|\det A| \cdot \widehat{V}_{\phi} (K, L^\circ).
 \ee
 Combining with equation (\ref{mixed volume-homogeneous-1}), one gets, for $\phi\in \widehat{\Phi}$,   \begin{eqnarray*} \widehat{V}_{\phi} (AK, \vrad(A^{-t} L)\cdot  (A^{-t} L)^\circ)&=& |\det A| \cdot |\det A^t|^{-1/n} \cdot \widehat{V}_{\phi} (K, \vrad(L)  L^\circ)\\ &=&|\det A|^{\frac{n-1}{n}}  \cdot \widehat{V}_{\phi} (K, \vrad(L)  L^\circ).  \end{eqnarray*}  The desired result follows immediately by taking the infimum over $L\in \cS_0$.  Other cases follow along the same lines.\end{proof}

Proposition \ref{affine-invariance:proposition} implies that both
$\affine(\cdot)$ and $\affineg(\cdot)$ are invariant under the volume
preserving linear transforms on $\bbR^n$. That is, for all $A\in
SL(n)$ and $K\in \cK_0$,   $$\affine(AK) =   \affine(K) \ \ \ \mathrm{and} \ \ \
\affineg(AK) =   \affineg(K).$$ In particular, $\affine(\lambda K) =\lambda ^{n-1} \cdot  \affine(K)$ and $\affineg(\lambda K) =\lambda ^{n-1} \cdot  \affineg(K)$ for $\lambda>0$ a constant.
This means that both $\affine(\cdot)$ and $\affineg(\cdot)$ have homogeneity.

 An immediate consequence of formula (\ref{ball-2}) and Proposition \ref{affine-invariance:proposition} is: for $\phi\in \widehat{\Phi}\cup \widehat{\Psi}$ and for the ellipsoid $\E=A\ball$ with $A\in GL(n)$,  $$\affine(\E)=\affineg(\E)=|\det A|^{\frac{n-1}{n}} \cdot  n\omega_n. $$

We can prove the following affine isoperimetric inequalities for the
homogeneous Orlicz $L_{\phi}$ affine and geominimal surface areas.

\bt  \label{homogeneous Orlicz affine isoperimetric inequality} Let  $K\in   \widetilde{\cK}$ be a convex body with its centroid or Santal\'{o} point at the origin.  \vskip 2mm
\noindent  (i)   For $\phi\in \widehat{\Phi}$, one has
\begin{eqnarray*}\frac{\affine(K)}{\affine(\ball)}\ \leq \  \frac{\affineg(K)}{\affineg(\ball)}\ \leq \
\bigg(\frac{|K|}{|\ball|}\bigg)^{\frac{n-1}{n}} \end{eqnarray*}
with equality  if and only if $K$ is an origin-symmetric ellipsoid.

\vskip 2mm \noindent (ii)  For $\phi\in \widehat{\Psi}$, there is a universal constant
$c>0$ such that  \begin{eqnarray*}\frac{\affine(K)}{\affine(\ball)} \ \geq \
\frac{\affineg(K)}{\affineg(\ball)}  \ \geq  c\cdot
\bigg(\frac{|K|}{|\ball|}\bigg)^{\frac{n-1}{n}}. \end{eqnarray*}
\et

\noindent {\bf Remark.} Theorem \ref{homogeneous Orlicz affine isoperimetric
inequality} asserts that among all convex bodies in $\widetilde{\cK}$ with fixed volume, the homogeneous Orlicz $L_{\phi}$ affine and geominimal
surface areas for $\phi\in \widehat{\Phi}$ attain their maximum at
origin-symmetric ellipsoids.  The $L_p$ affine isoperimetric inequalities for
the $L_p$ affine and geominimal surface areas  are special cases of Theorem \ref{homogeneous Orlicz affine isoperimetric
inequality} with $\phi(t)=t^p$ (see e.g., \cite{Lu1,
Petty1974, Petty1985, WY2008, Ye2015a}).

\begin{proof}  Formulas   (\ref{equal to
volume}) and (\ref{mixed volume-homogeneous-1})  together with Definition
\ref{definition:homogeneous:Orlicz:affine:surface} imply that  for
all $\phi\in \widehat{\Phi}$ and $K\in \cK_0$,
\begin{eqnarray}\label{comparision-1} \affine(K)\leq \affineg(K) \leq
\widehat{V}_{\phi} (K,\vrad(K^\circ)K)=n|K|\cdot \vrad(K^\circ).
\end{eqnarray} If $K\in \widetilde{\cK}$, the
Blaschke-Santal\'{o} inequality further implies,  for all
$\phi\in \widehat{\Phi}$,  \begin{eqnarray*}\affine(K) \leq \affineg(K) \leq
n|K|^{\frac{n-1}{n}}\cdot \omega_n^{1/n}  \end{eqnarray*} with
equality if and only if $K$ is an origin-symmetric ellipsoid (i.e.,
those make the equality hold in the Blaschke-Santal\'{o}
inequality). Dividing both sides by $\affine(\ball)=n\omega_n$, one gets the desired inequality in (i).

Similarly,  for all $\phi\in \widehat{\Psi}$ and for
all $K\in \cK_0$,
\begin{eqnarray}\label{comparision-2}\affine(K)\geq \affineg(K) \geq  n|K|\cdot
\vrad(K^\circ).   \end{eqnarray}  Dividing both sides by $\affine(\ball)=n\omega_n$, one gets
\begin{eqnarray*}\frac{\affine(K)}{\affine(\ball)} \ \geq \
\frac{\affineg(K)}{\affineg(\ball)}  \ \geq  c\cdot
\bigg(\frac{|K|}{|\ball|}\bigg)^{\frac{n-1}{n}}, \end{eqnarray*}  where the inverse Santal\'{o} inequality \cite{BM} has been used: there is a universal constant $c>0$ such that for all $K\in \widetilde{\cK}$, \begin{equation}\label{inverse-0717--0011} |K|\cdot |\polar|\geq c^n \omega_n^2. \end{equation}  See \cite{GK2, Nazarov2012} for estimates of the constant $c$.  \end{proof}

The following Santal\'{o} type inequalities follow immediately from Theorem \ref{homogeneous Orlicz affine isoperimetric inequality} and the Blaschke-Santal\'{o} inequality.
 \bt  \label{homogeneous Orlicz B-S inequality}  Let $K\in \widetilde{\cK}$ be a convex body with its centroid or Santal\'{o} point at the origin. \vskip 2mm \noindent  (i)  For $\phi\in \widehat{\Phi}$, one has  \begin{eqnarray*} \frac{\affine(K)\cdot  \affine(K^\circ)} { [{\affine(\ball)}]^2} \leq  \frac{ \affineg(K)\cdot  \affineg(K^\circ)}{ [{\affineg(\ball)}]^2} \leq  1.  \end{eqnarray*}  Equality holds if and only if $K$ is an origin-symmetric ellipsoid.

\vskip 2mm \noindent (ii) For $\phi\in \widehat{\Psi}$, there is a universal constant $c>0$ such that   \begin{eqnarray*}\frac{\affine(K)\cdot  \affine(K^\circ)} { [{\affine(\ball)}]^2} \geq  \frac{ \affineg(K)\cdot  \affineg(K^\circ)}{ [{\affineg(\ball)}]^2}  \geq  c^{n+1}.  \end{eqnarray*}   \et

A finer calculation could lead to stronger arguments than Theorem
\ref{homogeneous Orlicz affine isoperimetric inequality}, where the conditions on the centroid or the Santal\'{o}
point of $K$ can be removed. That is,  $\widetilde{\cK}$
in Theorem  \ref{homogeneous Orlicz affine isoperimetric
inequality} can be replaced by $\cK_0$. See similar results in \cite{Ye2013, Ye2015a, Ye2015b, Zhang2007}.
 \bc \label{homogeneous Orlicz affine isoperimetric inequality-0715}    Let  $K\in \cK_0$.  If either $\phi\in \widehat{\Phi}_1$  is  concave or $\phi\in \widehat{\Phi}_2$ is convex, then   \begin{eqnarray*}\frac{\affine(K)}{\affine(\ball)} \leq \frac{\affineg(K)}{\affineg(\ball)}\leq  \bigg(\frac{|K|}{|\ball|}\bigg)^{\frac{n-1}{n}}.  \end{eqnarray*}   In addition, if  either $\phi\in \widehat{\Phi}_1$ is strictly concave or $\phi\in \widehat{\Phi}_2$ is strictly convex, equality holds if and only if $K$ is an origin-symmetric ellipsoid.   \ec

To prove this corollary, one needs the following cyclic inequality. For convenience, let $H=\phi\circ\psi^{-1}$, where $\psi^{-1}$, the inverse of $\psi$,  always exists if $\psi\in \widehat{\Phi}\cup\widehat{\Psi}$.
 \bt\label{cyclic}
Let $K\in \cK_0$.  Assume one of the following conditions holds: a)
$\phi\in \widehat{\Phi}$ and $\psi\in \widehat{\Psi}$; b) $H$ is
convex with $\phi\in \widehat{\Phi}_2$ and $\psi\in
\widehat{\Phi}_1$; c)  $H$ is concave  with $\phi, \psi\in
\widehat{\Phi}_1$; d) $H$ is convex with either $\phi, \psi\in
\widehat{\Phi}_2$ or $\phi, \psi\in \widehat{\Psi}$. Then
\begin{equation*} \affine(K)\leq \widehat{\Omega}^{orlicz}_{\psi} (K) \ \ \ \mathrm{and} \ \ \ \affineg(K)\leq \widehat{G}^{orlicz}_{\psi} (K). \end{equation*} \et

\begin{proof}   The case for condition a) follows immediately from  formulas (\ref{comparision-1}) and (\ref{comparision-2}).  We only prove the case for condition b),  and the other cases follow along the same fashion. Assume that  condition b) holds and then $H$ is  convex. Corollary \ref{corollary:homogeneous-2} and Jensen's inequality  imply that  \begin{eqnarray*}
 1&=& \int _{S^{n-1}} \phi\Big(\frac{n|K|}{\widehat{V}_{\phi}(K, L^\circ)\cdot \rho_L(u) \cdot  h_K(u)}\Big)\dV \\ &=& \int _{S^{n-1}} H\bigg(\psi\Big(\frac{n|K|}{\widehat{V}_{\phi}(K, L^\circ) \cdot \rho_L(u) \cdot  h_K(u)}\Big)\bigg)\dV \\ &\geq& H\bigg(\int _{S^{n-1}} \psi\Big(\frac{n|K| }{\widehat{V}_{\phi}(K, L^\circ) \cdot \rho_L(u)\cdot  h_K(u)}\Big)\dV\bigg).  \end{eqnarray*} Together with  Corollary \ref{corollary:homogeneous-2} and the facts that  $H$ is decreasing and $H(1)=1$, one has
  \begin{eqnarray}\label{comparison-3}
\int _{S^{n-1}} \psi\Big(\frac{n|K|}{\widehat{V}_{\psi}(K, L^\circ)
\cdot \rho_L(u)\cdot  h_K(u)}\Big)\dV  \leq \int _{S^{n-1}}
\psi\Big(\frac{n|K|}{\widehat{V}_{\phi}(K, L^\circ)\cdot \rho_L(u)
\cdot  h_K(u)}\Big)\dV.  \ \ \end{eqnarray} Note that $\psi\in
\mathcal{I}$ (increasing). It follows from formula (\ref{comparison-3}) that
$\widehat{V}_{\phi}(K, L^\circ) \leq \widehat{V}_{\psi}(K,
L^\circ).$ Together with Corollary \ref{corollary:homogeneous-1}
and Definition \ref{definition:homogeneous:Orlicz:affine:surface}, one gets 
the desired result.\end{proof}

\vskip 2mm \noindent {\em Proof of Corollary \ref{homogeneous Orlicz
affine isoperimetric inequality-0715}.} Let $\psi(t)=t$ and $\phi\in
\widehat{\Phi}_2$ be convex. Then $H=\phi$ satisfies
condition b) in Theorem \ref{cyclic} and thus $ \affine(K)\leq
\widehat{\Omega}_{1}^{orlicz} (K).$ Note that
$\widehat{\Omega}_{1}^{orlicz} (K)$ is essentially the classical
geominimal surface area and is translation invariant. That is, for
any $z_0\in \bbR^n$, $\widehat{\Omega}_{1} ^{orlicz}
(K-z_0)=\widehat{\Omega}_{1}^{orlicz} (K)$. In particular, one
selects $z_0$ to be the point in $\bbR^n$ such that $K-z_0\in
\widetilde{\cK}$ (i.e., $z_0$ is either the centroid or the
Santal\'{o} point of $K$). Theorem \ref{homogeneous Orlicz affine isoperimetric inequality} implies that $$\frac{\affine(K)}{\affine(\ball)} \leq
\frac{\widehat{\Omega}_{1}^{orlicz}
(K-z_0)}{\widehat{\Omega}_{1}^{orlicz}(\ball)}\leq
\bigg(\frac{|K-z_0|}{|\ball|}\bigg)^{\frac{n-1}{n}}=\bigg(\frac{|K|}{|\ball|}\bigg)^{\frac{n-1}{n}}.
$$

To characterize the equality, due to the homogeneity of $\affine(\cdot)$, it is enough to prove that if $\phi$ is in addition strictly convex, ${\affine(K)}={\affine(\ball)}$  if and only if $K$ is an origin-symmetric ellipsoid with $|K|=\omega_n$. First of all,  if $K$ is an origin-symmetric ellipsoid with $|K|=\omega_n$, then ${\affine(K)}={\affine(\ball)}$ follows from Corollary \ref{equal-ball-0715} and Proposition \ref{affine-invariance:proposition}. On the other hand, by Theorem \ref{homogeneous Orlicz B-S inequality},  ${\affine(K)}={\affine(\ball)}$  holds only if $K-z_0$ is an origin-symmetric ellipsoid with $|K|=\omega_n$. By Proposition \ref{affine-invariance:proposition}, it is enough to claim  $K=\ball+z_0$ with $z_0=0$. Corollary \ref{equal-ball-0715} and Definition \ref{definition:homogeneous:Orlicz:affine:surface}  yield $$n\omega_n={\affine(\ball)}={\affine(K)}={\affine(\ball+z_0)}\leq \widehat{V}_{\phi}(\ball+z_0, \ball).$$  Note that  $\phi \in \widehat{\Phi}$ is convex and decreasing. Combining with Corollary \ref{corollary:homogeneous-2}, one has \begin{eqnarray*}  1 &=& \int _{S^{n-1}} \phi\bigg(\frac{n\omega_n}{ \widehat{V}_{\phi}(\ball+z_0, \ball) \cdot  h_{\ball+z_0}(u)}\bigg)\cdot \frac{h_{\ball+z_0}(u)}{n\omega_n}\cdot \,d\sigma(u)   \\&\geq & \phi\bigg(\int _{S^{n-1}} \frac{\,d\sigma(u)}{ \widehat{V}_{\phi}(\ball+z_0, \ball) } \bigg) \geq   1. \end{eqnarray*} As $\phi$ is strictly convex, equality holds if and only if $h_{\ball+z_0}(u)$ is a constant on $\sphere$. This yields $z_0=0$ as desired.

 The case for $\phi\in \widehat{\Phi}_1$ being concave (with characterization for equality) follows along the same lines. $\hfill \Box$

\section{The Orlicz-Petty bodies and the continuity of the homogeneous Orlicz geominimal surface areas}\label{section-homogeneous-petty}

This section concentrates on the continuity of the homogeneous Orlicz geominimal surface areas. In Subsection \ref{Semicontinuity of the homogeneous Orlicz geominimal surface areas}, we first show that the homogeneous Orlicz geominimal surface areas are semicontinuous on $\cK_0$ with respect to the Hausdorff distance. The existence and uniqueness of the Orlicz-Petty bodies, under certain conditions, will be proved in Subsection \ref{The Orlicz-Petty bodies: existence and basic properties}. Our main result on the continuity will be given in Subsection \ref{Continuity of the homogeneous Orlicz geominimal surface area}.

\subsection{Semicontinuity of the homogeneous Orlicz geominimal surface areas} \label{Semicontinuity of the homogeneous Orlicz geominimal surface areas}
Let us first establish the semicontinuity of the homogeneous Orlicz
geominimal surface areas.  Recall that for $\phi\in\widehat{\Phi}$ and for $K\in \cK_0$,
$$\widehat{G}_\phi^{orlicz}(K)
=\inf_{L\in\cK_0}\{\widehat{V}_\phi(K,\mathrm{vrad}(L)L^\circ)\}. $$ It is often more convenient, by the bipolar theorem (i.e., $(L^\circ)^\circ=L$ for $L\in \cK_0$) and Corollary \ref{corollary:homogeneous-1}, to formulate $\widehat{G}_\phi^{orlicz}(K)$ for $\phi\in\widehat{\Phi}$  by
\begin{eqnarray}
\label{equi-formula-geominimal-0715} \widehat{G}_\phi^{orlicz}(K)=\inf\{\widehat{V}_\phi(K, L):\ L\in\cK_0
\quad \mathrm{with}\quad |L^\circ|=\omega_n\}.
\end{eqnarray}  Similarly, for $\phi\in \widehat{\Psi}$, \begin{eqnarray} \widehat{G}_\phi^{orlicz}(K)=\sup\{\widehat{V}_\phi(K, L):\ L\in\cK_0
\quad \mathrm{with}\quad |L^\circ|=\omega_n\}.\label{equi-formula-geominimal-0715-1}
\end{eqnarray}

Denote by $r_K$ and $R_K$ the inner and outer radii of convex body $K\in\cK_0$, respectively. That is,  $$
r_K=\min\{h_K(u):\ u\in S^{n-1}\}\quad \mathrm{and} \quad
R_K=\max\{h_K(u):\ u\in S^{n-1}\}.
$$

 \bl\label{l-bounded}
 Let $K, L\in\mathcal{K}_0$. For
$\phi\in\mathcal{I}\cup\mathcal{D}$, one has
\begin{displaymath}
\frac{n\omega_n\cdot r_K^n\cdot r_{L}}{R_K}\leq
\widehat{V}_\phi(K,L)\le \frac{n\omega_n\cdot R_K^n\cdot
R_{L}}{r_K}.
\end{displaymath}
 \el

\begin{proof}
For $\phi \in \mathcal{I}$, let $\lambda=\widehat{V}_\phi(K,L)$. By Corollary \ref{corollary:homogeneous-2} and the
fact that $\phi$ is increasing on $(0, \infty)$, one has
\begin{eqnarray*}
1&=&  {\int}_{S^{n-1}}\phi\left(\frac{n|K|\cdot
h_{L}(u)}{\lambda\cdot h_K(u)}\right)
d\widetilde{V}_K(u)\\
&\leq&   {\int}_{S^{n-1}}\phi\left(\frac{n|K|\cdot
R_{L}}{\lambda\cdot r_K}\right)
d\widetilde{V}_K(u)  \nonumber\\
&\leq& \phi\left(\frac{n\omega_n\cdot R_K^n\cdot R_{L}}{\lambda\cdot
r_K}\right). \nonumber
\end{eqnarray*}
Moreover, as $\phi(1)=1$,  one gets
\begin{equation*}
\widehat{V}_\phi(K,L)=\lambda\le \frac{n \omega_n\cdot R_K^n\cdot R_{L}}{r_K}.
\end{equation*}
For the lower bound,
\begin{eqnarray*}
1&=&  {\int}_{S^{n-1}}\phi\left(\frac{n|K|\cdot
h_{L}(u)}{\lambda\cdot h_K(u)}\right)
d\widetilde{V}_K(u)\\
&\geq&   {\int}_{S^{n-1}}\phi\left(\frac{n|K|\cdot
r_{L}}{\lambda\cdot R_K}\right)
d\widetilde{V}_K(u)  \nonumber\\
&\geq& \phi\left(\frac{n\omega_n\cdot r_K^n\cdot r_{L}}{\lambda\cdot
R_K}\right). \nonumber
\end{eqnarray*} As  $\phi$ is increasing on $(0,
\infty)$ and $\phi(1)=1$,   one gets
\begin{equation*}
\widehat{V}_\phi(K,L)\geq \frac{n\omega_n\cdot r_K^n \cdot r_{L}}{R_K}.
\end{equation*}

The case for $\phi\in \mathcal{D}$ follows along the same lines. \end{proof}

We will often need the following result.
\bl \label{uniform-continuous-convergence} Let $\varphi: I\rightarrow \bbR$ be a uniformly continuous function on an interval $I\subset \bbR$. Let $\{f_i\}_{i\geq 0}$ be a sequence of functions such that $f_i: E\rightarrow I$ for all $i\geq 0$ and $f_i\rightarrow f_0$ uniformly  on $E$ as $i\rightarrow \infty$. Then $\varphi(f_i)\rightarrow \varphi(f_0)$  uniformly on $E$ as $i\rightarrow \infty$.  \el \begin{proof} For any $\epsilon>0$. As $\varphi$ is uniformly continuous, there exists $\delta(\epsilon)>0$ such that $|\varphi(x)-\varphi(y)|<\epsilon$ for all $x, y\in I$ with $|x-y|<\delta(\epsilon)$.   
On the other hand, as $f_i\rightarrow f_0$ uniformly on $E$, there exists an integer $N_0(\epsilon): =N(\delta(\epsilon))>0$ such that
$|f_i(z)-f_0(z)|<\delta(\epsilon)$ for all $i>N_0(\epsilon)$ and all $z\in E$. Hence, $|\varphi(f_i(z))-\varphi(f_0(z))|<\epsilon$ for all $i>N_0(\epsilon)$ and all $z\in E$.  That is, $\varphi(f_i)\rightarrow \varphi(f_0)$ uniformly on $E$.
\end{proof}

\bp\label{p2} Let $\{K_i\}_{i\geq 1}$ and $\{L_i\}_{i\geq 1}$ be two sequences of convex bodies in $\cK_0$ such that  $K_i\rightarrow K\in\mathcal{K}_{0}$ and
$L_i\rightarrow L\in\mathcal{K}_{0}$. For $\phi\in \mathcal{I}\cup \mathcal{D}$, one has $\widehat{V}_\phi(K_i,L_i)\rightarrow
\widehat{V}_\phi(K, L).$  \ep

\begin{proof} As  $K_i\rightarrow K\in\mathcal{K}_{0}$, one can find constants $c_K, C_K>0$, such that, for all $i\geq 1$, \begin{equation}\label{bounded-1-0716} c_K \ball \subset K_i, K \subset C_K\ball. \end{equation}  Similarly, one can find constants $c_L, C_L>0$, such that, for all $i\geq 1$,  \begin{equation}\label{bounded-2-0716}c_L\ball \subset L_i, L \subset C_L\ball. \end{equation}

For simplicity,  let
$\lambda_i=\widehat{V}_\phi(K_i,L_i)$. Lemma \ref{l-bounded}
yields, for all $i\geq 1$, 
\begin{equation}
\frac{n\omega_n\cdot c_K^n\cdot c_L}{C_K}\leq
\lambda_i\le \frac{n\omega_n\cdot C_{K}^n\cdot C_L} {c_K}, \label{uniform-mixed-bound-0715}
\end{equation} and thus   the sequence
$\{\lambda_i\}_{i\geq 1}$ is bounded from both sides. Let $f_i$ and $f$ be given by  $$f_i(u)=\frac{n|K_i|\cdot h_{L_i}(u)}{\lambda_i \cdot h_{K_i}(u)} \ \ \ \mathrm{and}\ \ \ f(u)= \frac{n|K|\cdot h_L(u)}{\lambda_0\cdot h_K(u)}\ \ \ \mathrm{for} \ \ u\in \sphere. $$

On the one hand, suppose that $\{\lambda_{i_k}\}_{k\geq 1}$ is a convergent subsequence of $\{\lambda_i\}_{i\geq 1}$ with limit $\lambda_0$. That is, $\lim_{k\rightarrow \infty} \lambda_{i_k}=\lambda_0$ and then $0<\lambda_0<\infty$. Note that $K_i\rightarrow K\in \cK_0$ yields $h_{K_i}\rightarrow h_{K}$ uniformly on $S^{n-1}$.  Similarly,   $h_{L_i}\rightarrow h_{L}$ uniformly on $S^{n-1}$. Together with (\ref{bounded-1-0716}) and (\ref{bounded-2-0716}), one sees that $f_{i_k} \rightarrow
f$ uniformly on $S^{n-1}.$  Moreover, the ranges of $f_{i_k}, f$ are all in the interval $$I=\bigg[ \frac{c_L}{C_L} \cdot \bigg(\frac{c_K}{C_K}\bigg)^{n+1},\ \     \frac{C_L}{c_L} \cdot \bigg(\frac{C_K}{c_K}\bigg)^{n+1}\bigg].$$ Note that the interval $I\subsetneq (0, \infty)$ is a compact set. Hence $\phi\in\mathcal{I}\cup\mathcal{D}$ restricted on $I$ is uniformly continuous.  Lemma \ref{uniform-continuous-convergence}  implies that $\phi(f_{i_k})\rightarrow \phi(f)$ uniformly on $\sphere$. Moreover, as both $\{\phi(f_{i_k})\}_{k\geq 1}$ and $\{h_{K_i}\}_{i\geq 1}$ are uniformly bounded on $\sphere$, one sees that  $\phi(f_{i_k})h_{K_{i_k}} \rightarrow \phi(f)h_K$ uniformly on $\sphere$. Formula (\ref{weak fact}) then yields
\begin{eqnarray*}
1&=&\lim_{k\rightarrow\infty}  {\int}_{S^{n-1}}\phi\left(\frac{n|K_{i_k}|\cdot
h_{L_{i_k}}(u)}{\lambda_{i_k}\cdot h_{K_{i_k}}(u)}\right)
\, d\widetilde{V}_{K_{i_k}}(u)\\ &=& \lim_{k\rightarrow\infty}  {\int}_{S^{n-1}}   \frac{\phi\left(f_{i_k}(u)\right)h_{K_{i_k}}(u)  }{n |K_{i_k}|} \, dS_{K_{i_k}}(u)  \\ &=&  {\int}_{S^{n-1}}  \frac{\phi\left(f(u)\right)h_{K}(u)}{n |K|} \, dS_{K}(u)\\ &=&{\int}_{S^{n-1}}  {\phi\left(\frac{n|K|\cdot h_L(u)}{\lambda_0\cdot h_K(u)}\right)} \, d\widetilde{V}_{K}(u).
\end{eqnarray*} Therefore $\lambda_0=\widehat{V}_\phi(K, L)$ and $\lim_{k\rightarrow \infty} \widehat{V}_\phi(K_{i_k}, L_{i_k}) =\widehat{V}_\phi(K, L).$ We have proved that if a subsequence of  $\{\widehat{V}_\phi(K_i,L_i)\}_{i\geq 1}$ is convergent, then its limit must be $\widehat{V}_\phi(K, L)$.

To conclude Proposition \ref{p2}, it is enough to claim that the
sequence $\{\widehat{V}_\phi(K_i,L_i)\}_{i\geq 1}$ is indeed convergent.
Suppose that $\{\widehat{V}_\phi(K_i,L_i)\}_{i\geq 1}$ is not
convergent. One has two convergent subsequences whose limits exist  by
(\ref{uniform-mixed-bound-0715}) and are different.  This
contradicts with the arguments in the previous paragraph, and hence
the sequence $\{\widehat{V}_\phi(K_i,L_i)\}_{i\geq 1}$ is
convergent. \end{proof}

The following result states that the homogeneous Orlicz geominimal
surface areas are semicontinuous. For the homogeneous Orlicz affine
surface areas, similar semicontinuous arguments also hold.

\bp\label{semicontinuous-07-17-1----1} For $\phi\in \widehat{\Phi}$, the functional $\affineg(\cdot)$ is upper semicontinuous on $\cK_0$  with respect to the Hausdorff distance. That is, for any convergent sequence $\{K_i\}_{i\geq 1}\subset\cK_0$ whose limit is $K_0\in \cK_0$, then $$\affineg(K_0)\geq \limsup_{i\rightarrow\infty} \affineg(K_i).$$ While for $\phi\in \widehat{\Psi}$, the functional $\affineg(\cdot)$ is lower semicontinuous on $\cK_0$: for any $K_i\rightarrow K_0$, then $$\affineg(K_0)\leq \liminf_{i\rightarrow\infty} \affineg(K_i).$$  \ep
\begin{proof} Let $\phi\in \widehat{\Phi}$. For any given $\epsilon>0$, by formula (\ref{equi-formula-geominimal-0715}), there exists a convex body $L_{\epsilon}\in \cK_0$, such that $|L_{\epsilon}^\circ|=\omega_n$ and $$\affineg(K_0)+\epsilon>\widehat{V}_{\phi}(K_0, L_{\epsilon})\geq \affineg(K_0). $$ By Proposition \ref{p2}, one has
\begin{eqnarray*}  \affineg(K_0)+\epsilon > \widehat{V}_{\phi}(K_0, L_{\epsilon}) = \lim_{i\rightarrow \infty} \widehat{V}_{\phi}(K_i, L_{\epsilon})= \limsup_{i\rightarrow \infty} \widehat{V}_{\phi}(K_i, L_{\epsilon})\geq  \limsup_{i\rightarrow\infty} \affineg(K_i).
\end{eqnarray*}
The desired result follows by letting $\epsilon\rightarrow 0$. The case for $\phi\in \widehat{\Psi}$ can be proved along the same lines. \end{proof}

 \subsection{The Orlicz-Petty bodies: existence and basic properties} \label{The Orlicz-Petty bodies: existence and basic properties}
In this subsection, we will prove the existence of the Orlicz-Petty
bodies under the condition $\phi\in \widehat{\Phi}_1$. The following
lemma is needed for our goal. Denote by $a_+=\frac{a+|a|}{2}$ for $a\in \bbR$. Clearly $a_+=\max\{a, 0\}$.

\bl \label{projection-positive-0715} Let $K\in \cK_0$ and $\phi\in\widehat{\Phi}_1$. For fixed $v\in \sphere$, define  $G_v: (0, \infty)\rightarrow (0, \infty)$ by $$G_v(\eta)=\int_{\sphere}   \phi\left(\frac{n|K|\cdot\langle
u, v\rangle_+}{\eta \cdot h_K(u)}\right)\,d\widetilde{V}_K(u).$$ Then $G_v$  is strictly decreasing, and
\begin{eqnarray*} \lim_{\eta\rightarrow 0} G_v(\eta)=\lim_{t\rightarrow \infty} \phi(t)=\infty \ \ \ \mathrm{and}\ \ \  \lim_{\eta\rightarrow \infty} G_v(\eta)= \lim_{t\rightarrow  0} \phi(t)=0.  \end{eqnarray*} \el

\begin{proof}  Since $K\in \mathcal{K}_0$, (\ref{minkowski-solution-1}) implies that there exists a  constant
$c_1>0$ such that for all $v\in\sphere$,  $$\int_{S^{n-1}}\langle u,v\rangle_+dS_K(u)\geq c_1.$$
For any given $v\in \sphere$, let $\Sigma_j(v)=\{u\in S^{n-1}:\langle u,v\rangle_+ >\frac{1}{j}\}$ for all integers $j\geq 1$.  It is obvious that $\Sigma_j(v)\subset \Sigma_{j+1}(v)$ for all $j\geq 1$ and $\cup_{j=1}^{\infty}\Sigma_j(v)=\{u\in S^{n-1}:\langle u,v\rangle_+ >0\}$. Hence,
$$\lim_{j\rightarrow \infty}\int_{\Sigma_j(v)}\langle u,v\rangle_+ \,dS_K(u)=\int_{\cup_{j=1}^{\infty}\Sigma_j(v)}\langle u,v\rangle_+ \,dS_K(u)=\int_{S^{n-1}}\langle u,v\rangle_+ \,dS_K(u)\geq c_1.
$$ Then, there exists an integer  $j_0\geq 1$ (depending on $v\in \sphere$)
such that \be
\label{sphere integration} \frac{c_1}{2}\leq \int_{\Sigma_{j_0}(v)}\langle u,v\rangle_+ \,
dS_K(u)\leq \int_{\Sigma_{j_0}(v)}\,dS_K(u). \ee

Assume that $\phi\in \widehat{\Phi}_1$ and then $\phi$ is strictly increasing.   Let $0<\eta_1<\eta_2<\infty$. For all  $u\in \Sigma_{j_0} (v)$, one has  $$\phi\left(\frac{n|K|\cdot\langle
u, v\rangle_+}{\eta_2 \cdot h_K(u)}\right)<\phi\left(\frac{n|K|\cdot\langle
u, v\rangle_+}{\eta_1 \cdot h_K(u)}\right),$$ and by (\ref{sphere integration}),  $$\int_{\Sigma_{j_0} (v)}   \phi\left(\frac{n|K|\cdot\langle
u, v\rangle_+}{\eta_2 \cdot h_K(u)}\right)\,d\widetilde{V}_K(u)< \int_{\Sigma_{j_0} (v)}   \phi\left(\frac{n|K|\cdot\langle
u, v\rangle_+}{\eta_1 \cdot h_K(u)}\right)\,d\widetilde{V}_K(u).$$ The desired monotone argument (i.e., $G_v$ is strictly decreasing) follows immediately from
\begin{eqnarray*} G_v(\eta)= \int_{\Sigma_{j_0} (v)}   \phi\left(\frac{n|K|\cdot\langle
u, v\rangle_+}{\eta \cdot h_K(u)}\right)\,d\widetilde{V}_K(u)+ \int_{\sphere\setminus \Sigma_{j_0} (v)}   \phi\left(\frac{n|K|\cdot\langle
u, v\rangle_+}{\eta \cdot h_K(u)}\right)\,d\widetilde{V}_K(u).  \end{eqnarray*}

Now let us prove that  \begin{eqnarray*} \lim_{\eta\rightarrow 0} G_v(\eta)=\lim_{t\rightarrow \infty} \phi(t) =\infty \ \ \ \mathrm{and}\ \ \   \lim_{\eta\rightarrow \infty} G_v(\eta)= \lim_{t\rightarrow  0} \phi(t)=0.  \end{eqnarray*} To this end,  as $\phi\in \widehat{\Phi}_1$ is increasing, \begin{eqnarray*} G_v(\eta)=\int_{\sphere}   \phi\left(\frac{n|K|\cdot\langle
u, v\rangle_+}{\eta \cdot h_K(u)}\right)\,d\widetilde{V}_K(u) \leq \int_{\sphere}   \phi\left(\frac{n|K| }{\eta \cdot r_K}\right)\,d\widetilde{V}_K(u)=\phi\left(\frac{n|K| }{\eta \cdot r_K}\right).\end{eqnarray*} By letting $t=\frac{n|K| }{\eta \cdot r_K}$, one has $0\leq \lim_{\eta\rightarrow \infty} G_v(\eta)\leq \lim_{t\rightarrow  0} \phi(t)=0$ and thus $ \lim_{\eta\rightarrow \infty} G_v(\eta)=0$.   On the other hand, \begin{eqnarray} G_v(\eta)&=&\int_{\sphere}   \phi\left(\frac{n|K|\cdot\langle
u, v\rangle_+}{\eta \cdot h_K(u)}\right)\,d\widetilde{V}_K(u) \nonumber\\  &\geq& \int_{\Sigma_{j_0} (v)}  \phi\left(\frac{n|K|\cdot\langle
u, v\rangle_+}{\eta \cdot h_K(u)}\right)\,d\widetilde{V}_K(u) \nonumber \\ &\geq&  \int_{\Sigma_{j_0} (v)}  \phi\left(\frac{n|K| }{\eta\cdot j_0 \cdot R_K}\right)\cdot\frac{r_K}{n|K|}\cdot \,dS_K(u) \nonumber \\ &\geq &  \phi\left(\frac{n|K| }{\eta\cdot j_0 \cdot R_K}\right)\cdot \frac{r_K}{n|K|}\cdot \frac{c_1}{2}.\label{projection-estimate-0717}\end{eqnarray} The desired result $\lim_{\eta\rightarrow 0} G_v(\eta)= \lim_{t\rightarrow \infty} \phi(t)=\infty$ follows by taking $\eta\rightarrow 0$. \end{proof} A direct consequence of Lemma \ref{projection-positive-0715} is that if $\phi \in \widehat{\Phi}_1$ and $v\in \sphere$, then there is a unique $\eta_0\in (0, \infty)$ such that \begin{eqnarray*} G_v(\eta_0)=\int_{\sphere}   \phi\left(\frac{n|K|\cdot\langle
u, v\rangle_+}{\eta_0 \cdot h_K(u)}\right)\,d\widetilde{V}_K(u)=1.   \end{eqnarray*}   Such a unique $\eta_0$ can be defined as the homogeneous Orlicz $L_{\phi}$ mixed volume of $K\in \cK_0$ and the line segment $[0, v]=\{tv: t\in [0, 1]\}$, namely,  $\eta_0=\widehat{V}_{\phi}(K, [0, v])$ and
\begin{eqnarray}  \int_{\sphere}   \phi\left(\frac{n|K|\cdot\langle
u, v\rangle_+}{\widehat{V}_{\phi}(K, [0, v])\cdot h_K(u)}\right)\,d\widetilde{V}_K(u)=1.  \label{projection-orlicz-0715---1} \end{eqnarray}

\bp \label{p3} Let $K\in\mathcal{K}_{0}$ and $\phi\in\widehat{\Phi}_1$. There exists a convex body
$M\in\mathcal{K}_{0}$ such that
\begin{equation*}
\widehat{G}^{orlicz}_\phi(K)=\widehat{V}_\phi(K, M) \quad
\mathrm{and}\ \quad |M^\circ|=\omega_n.
\end{equation*}
If in addition $\phi$ is convex,  such a convex body $M$ is unique.
\ep

\begin{proof} Formula (\ref{equi-formula-geominimal-0715}) implies that for
$\phi\in\widehat{\Phi}_1$, there exists a sequence
$\{M_i\}_{i\geq 1} \subset \mathcal{K}_{0}$ such that  $\widehat{V}_\phi(K,M_i)  \rightarrow \widehat{G}^{orlicz}_\phi(K)$ as $i\rightarrow \infty$,  $|M_i^\circ|=\omega_n$ and $ \widehat{V}_\phi(K,M_i) \leq 2\widehat{V}_\phi(K,\ball)$  for
all $ i\geq 1$.   For each fixed $i\geq 1$, let $$R_i=\rho_{M_i}(u_i)=\max\{\rho_{M_i}(u):u\in S^{n-1}\}.$$ This
yields $\{\lambda u_i: 0\le\lambda\le R_i\}\subset M_i$ and hence for all $u\in \sphere$,
$$
h_{M_i}(u)\ge R_i\cdot \frac{|\langle u,u_i\rangle|+\langle
u,u_i \rangle}{2}=R_i\cdot\langle u,u_i\rangle_+.
$$

 Let  $\phi\in \widehat{\Phi}_1$ and $\eta_i=\widehat{V}_{\phi}(K, [0, u_i])\in (0, \infty)$ for $i\geq 1$. Recall that formula (\ref{projection-orlicz-0715---1}) states
  \begin{eqnarray*}
1= \int_{S^{n-1}}\phi\left(\frac{n|K|\cdot\langle
u,u_i\rangle_+}{\eta_i\cdot h_K(u)}\right)\,d\widetilde{V}_K(u). \end{eqnarray*}
  By Corollary
\ref{corollary:homogeneous-2} and the fact that $\phi$ is increasing, we have
\begin{eqnarray*}
1&=& \int_{S^{n-1}}\phi\left(\frac{n|K|\cdot h_{M_i}(u)}{\widehat{V}_\phi(K,M_i)\cdot h_K(u)}\right)\,d\widetilde{V}_K(u)\\
&\geq& \int_{S^{n-1}}\phi\left(\frac{n|K|\cdot R_i\cdot\langle
u,u_i\rangle_+}{2\widehat{V}_\phi(K,B_2^n)\cdot
h_K(u)}\right)\, d\widetilde{V}_K(u).
\end{eqnarray*} This further leads to, for all $i\geq 1$,
$$R_i\leq \frac{2 \widehat{V}_\phi(K,B_2^n)}{\eta_i}.$$

Next, we prove that $\inf_{i\geq 1}\eta_i>0$. We will use the method
of contradiction and assume that   $\inf_{i\geq 1} \eta_i= 0$.
Consequently, there is a subsequence of $\{\eta_i\}_{i\geq 1}$ (still denoted by
$\{\eta_i\}_{i\geq 1}$), such that, $\eta_i\rightarrow 0$ as
$i\rightarrow \infty$.  Due to the compactness of $\sphere$, one
can also have a convergent subsequence of $\{u_i\}_{i\geq 1}$ (again denoted by $\{u_i\}_{i\geq 1}$) whose limit is $v\in \sphere$.
In summary, we have two sequences  $\{u_i\}_{i\geq 1}$ and
$\{\eta_i\}_{i\geq 1}$ such that $u_i\rightarrow v$  and
$\eta_i\rightarrow 0$ as $i\rightarrow \infty$.  It is easily
checked that $\langle u,u_i\rangle_+ \rightarrow \langle
u,v\rangle_+$ uniformly on $S^{n-1}$ by the triangle inequality. For any
given $\varepsilon>0$, Corollary \ref{corollary:homogeneous-2},
Fatou's lemma and formula (\ref{sphere integration}) imply
\begin{eqnarray*}
1&=& \lim_{i\rightarrow \infty}\int_{S^{n-1}}\phi\left(\frac{n|K|\cdot \langle u,u_i\rangle_+}{\eta_i\cdot h_K(u)}\right)\dV\\
&\geq& \liminf_{i\rightarrow \infty}\int_{S^{n-1}}\phi\left(\frac{n|K|\cdot \langle u,u_i\rangle_+}{(\eta_i+\varepsilon)\cdot h_K(u)}\right)\dV\\&\geq& \int_{S^{n-1}}\liminf_{i\rightarrow \infty} \phi\left(\frac{n|K|\cdot \langle u,u_i\rangle_+}{(\eta_i+\varepsilon)\cdot h_K(u)}\right)\dV\\
&=& \int_{S^{n-1}}\phi\left(\frac{n|K|\cdot \langle u,v\rangle_+}{\varepsilon\cdot h_K(u)}\right)\dV \\
&=& G_{v}(\varepsilon).
\end{eqnarray*} It follows from Lemma \ref{projection-positive-0715} that $\lim_{\varepsilon\rightarrow 0^+}G_v(\varepsilon)=\infty$, which leads to  a contradiction (i.e., $1\geq \infty$). Therefore, $\inf_{i\geq 1}\eta_i>0$ and
$$
\sup_{i\geq 1} R_i\leq \frac{2\widehat{V}_\phi(K,B_2^n)}{\inf_{i\geq 1}\eta_i}<\infty.
$$ This concludes that the sequence $\{M_i\}_{i\geq 1}\subset \cK_0$  is uniformly bounded.

The Blaschke selection
theorem yields that there exists a convergent subsequence of $\{M_i\}_{i\geq 1}$ (still
denoted by $\{M_i\}_{i\geq 1}$) and a convex body $M\in\cK$ such that
$M_i\rightarrow M$ as $i\rightarrow \infty$. Since $|M_i^\circ|=\omega_n$ for all $i\geq 1$,  Lemma \ref{l1}
implies $M\in\cK_0$. Moreover, $|M^\circ|=\omega_n$ because $|M_i^\circ|=\omega_n$ for all $i\geq 1$ and $M_i\rightarrow M$ (hence, 
$M_i^\circ\rightarrow M^\circ$).  It follows from Proposition \ref{p2} that $$\widehat{V}_\phi(K,M_i)  \rightarrow  \widehat{V}_\phi(K,M) \ \ \ \mathrm{and} \ \ \ |M^\circ|=\omega_n. $$ By the uniqueness of the limit, one gets
$$  \widehat{G}^{orlicz}_\phi(K)=\widehat{V}_\phi(K,M) \ \ \ \mathrm{and} \ \ \ |M^\circ|=\omega_n.$$
This concludes the existence of the Orlicz-Petty bodies.

  If $\phi\in \widehat{\Phi}_1$ is also convex, the
uniqueness of $M$ can be proved as follows. Suppose that ${M}_1, {M}_2 \in \cK_0$
such that $|M_1^\circ|=|{M}_2 ^\circ|=\omega_n$ and
$$\widehat{V}_\phi(K, M_1)=\inf_{L\in\cK_0}\{\widehat{V}_\phi(K,\mathrm{vrad}(L^\circ)L)\}=\widehat{V}_\phi(K, {M}_2).$$
Define $M\in\cK_0$ by $M=\frac{M_1+M_2}{2}.
$ That is, $h_M=\frac{h_{M_1}+h_{M_2}}{2}$.
By formula (\ref{formula for volume}), it  can be checked  that
$|M^\circ|\leq\omega_n$ (hence $\mathrm{vrad}(M^\circ)\leq 1$) with equality if and only if $M_1=M_2$. In fact, the function $t^{-n}$ is strictly convex, and hence \begin{eqnarray} |M^\circ| &=&\frac{1}{n} \int _{\sphere} h_M(u)^{-n}\,d\sigma(u) \nonumber \\ &=& \frac{1}{n} \int _{\sphere}  \bigg(\frac{h_{M_1}(u)+h_{M_2}(u)}{2}\bigg)^{-n} \,d\sigma(u) \nonumber \\ &\leq& \frac{1}{n} \int _{\sphere}  \frac{h_{M_1}(u)^{-n}+h_{M_2}(u)^{-n}}{2} \,d\sigma(u) \nonumber\\ &=& \frac{|M_1^\circ|+|M_2^\circ|}{2} =\omega_n,  \label{mean-body-volume} \end{eqnarray} with equality  if and only if $h_{M_1}=h_{M_2}$ on $\sphere$, i.e., $M_1=M_2$.

For convenience, let $\lambda=\widehat{V}_\phi(K,M_1)=\widehat{V}_\phi(K,M_2)$. The fact that $\phi$ is convex imply \begin{eqnarray*}
\int_{S^{n-1}}\phi\left(\frac{n|K|\cdot h_M(u)}{\lambda\cdot h_K(u)}\right)\dV\! \!&=&\!\! \int_{S^{n-1}}\phi\left(\frac{n|K|\cdot (h_{M_1}(u)+h_{M_2}(u))}{2\cdot \lambda\cdot h_K(u)}\right)\dV\\ \!\! &\leq&\!\! \frac{1}{2}\int_{S^{n-1}}\bigg[\phi\left(\frac{n|K| \cdot h_{M_1}(u)}{\lambda \cdot h_K(u)}\right)+\phi\left(\frac{n|K| \cdot h_{M_2}(u)}{\lambda\cdot h_K(u)}\right)\bigg]\dV\\
\!\!&=&\!\! 1. \end{eqnarray*}
Hence, $\widehat{V}_\phi(K,M)\leq \lambda$ which follows from the facts that $\phi$ is strictly increasing and
\begin{eqnarray*}
\int_{S^{n-1}}\phi\left(\frac{n|K|\cdot h_M(u)}{\lambda\cdot h_K(u)}\right)\dV \leq 1
= \int_{S^{n-1}}\phi\left(\frac{n|K|\cdot h_M(u)}{\widehat{V}_\phi(K, M) \cdot h_K(u)}\right)\dV.
\end{eqnarray*}

Assume that $M_1\neq M_2$, then $\mathrm{vrad}(M^\circ)<1$. Note that $\widehat{V}_\phi(K,M)>0.$ Together with  Corollary \ref{corollary:homogeneous-1}, one can check that $$\widehat{V}_\phi(K,\mathrm{vrad}(M^\circ)M)<\widehat{V}_\phi(K,M)\leq \widehat{V}_\phi(K,M_1).$$ This contradicts with the minimality of $M_1$. Therefore, $M_1=M_2$ and the uniqueness follows.
\end{proof}

\bd \label{orlicz-petty body} Let $K\in\cK_{0}$ and
$\phi\in\widehat{\Phi}_1$. A convex body $M$ is said to be an $L_{\phi}$
Orlicz-Petty body of $K$, if $M\in \cK_0$ satisfies  $$
\widehat{G}^{orlicz}_\phi(K)=\widehat{V}_\phi(K,M)  \quad
\mathrm{and} \quad |M^\circ|=\omega_n. $$ Denote by
$\widehat{T}_\phi K$ the set of all $L_{\phi}$ Orlicz-Petty bodies
of $K$.  \ed

Clearly, if $\phi\in \widehat{\Phi}_1$, the set $\widehat{T}_\phi K
$ is nonempty and may contain more than one convex body. If in
addition $\phi\in \widehat{\Phi}_1$ is convex, $\widehat{T}_\phi K $
must contain only one convex body; and in this case $\widehat{T}_{\phi}K$ is called the $L_\phi$ Orlicz-Petty body of $K$. Moreover, the set
$\widehat{T}_{\phi} K$ is $SL(n)$-invariant. In fact, for $A\in
SL(n)$ and all $M\in \widehat{T}_\phi K$, by Proposition
\ref{affine-invariance:proposition} and  formula (\ref{transform}),
one sees
$$\affineg(AK)=\affineg(K)=\widehat{V}_\phi(K,M)=\widehat{V}_\phi(AK,AM).$$
It follows from $|(AM)^\circ|=\omega_n$ that $AM\in \widehat{T}_\phi
(AK)$ and thus $A(\widehat{T}_\phi K)\subset \widehat{T}_\phi (AK).$
Replacing $K$ by $AK$ and  $A$ by its inverse, one also gets
$\widehat{T}_\phi (AK)\subset A (\widehat{T}_\phi K)$ and thus 
$\widehat{T}_\phi (AK)=A(\widehat{T}_\phi K)$.

On the other hand, $\widehat{T}_{\phi}(\lambda
K)=\widehat{T}_{\phi}K$ for all $\lambda>0$. To this end, for $M\in
\widehat{T}_{\phi} K$, one has $|M^\circ|=\omega_n$ and
$\affineg(K)=\widehat{V}_\phi(K,M).$  This leads to, by Corollary
\ref{corollary:homogeneous-1} and Proposition
\ref{affine-invariance:proposition}, $$\affineg(\lambda
K)=\lambda^{n-1} \affineg(K) =\lambda^{n-1}\widehat{V}_\phi(K,
M)=\widehat{V}_\phi(\lambda K, M).$$ Thus, $M\in
\widehat{T}_{\phi}(\lambda K)$ and then $\widehat{T}_{\phi} K\subset
\widehat{T}_{\phi}(\lambda K)$. Similarly,
$\widehat{T}_{\phi}(\lambda K)\subset \widehat{T}_{\phi} K$ and thus
$\widehat{T}_{\phi}(\lambda K)=\widehat{T}_{\phi}K$.

When $\phi\in \widehat{\Phi}_1$ is convex, the $L_{\phi}$ Orlicz-Petty body $\widehat{T}_{\phi}K$ satisfies the following inequality: for all $K\in \cK_0$, one has \begin{eqnarray}
|\widehat{T}_\phi K|\cdot |(\widehat{T}_\phi K)^\circ| \leq |K|\cdot |K^\circ|. \label{Mahler-product-0717--1} \end{eqnarray} In fact, it follows from (\ref{comparision-1}) that  for $K\in\cK_0$, $
\widehat{G}^{orlicz}_\phi(K)\le
n|K|\cdot \vrad(\polar).$  Definition \ref{orlicz-petty body}
and the Orlicz-Minkowski inequality  (\ref{Minkowski-inequality-2016-713}) imply that
$$\widehat{G}^{orlicz}_\phi(K)=\widehat{V}_\phi(K,\widehat{T}_\phi K) \geq  n\cdot |K|^{\frac{n-1}{n}} |\widehat{T}_\phi
K|^\frac{1}{n}.
$$
The desired inequality (\ref{Mahler-product-0717--1}) is then a
simple consequence of the combination of the two inequalities above
and $|(\widehat{T}_\phi K)^\circ|=\omega_n$. Note that it is an open
problem (i.e., the famous Mahler conjecture)  to find the minimum of
$|K|\cdot|\polar|$ among all convex bodies $K\in \widetilde{\cK}$.
The inverse Santal\'{o} inequality (\ref{inverse-0717--0011})
provides an isomorphic solution to the Mahler conjecture. We think
that the $L_{\phi}$ Orlicz-Petty body $\widehat{T}_{\phi}K$ and
inequality (\ref{Mahler-product-0717--1}) may be useful in attacking
the Mahler conjecture.

The following proposition states that an $L_\phi$ Orlicz-Petty body of a polytope is again a polytope.

 \bp Let $K\in\cK_0$ be a polytope and $\phi \in \widehat{\Phi}_1$.  If $M\in \widehat{T}_\phi
K$, then $M$ is a polytope with faces parallel to those of $K$.   \ep

\begin{proof} Let  $K$ be a polytope whose surface area measure $S_K$ is concentrated on a finite set $\{u_1, \cdots, u_m\}\subset \sphere$.  Let $M\in \widehat{T}_{\phi}K$ be an $L_{\phi}$ Orlicz-Petty body of $K$. Denote by $P$ the polytope whose faces are parallel to those of $K$ and $P$ circumscribes $M$.

Note that  $S_K$ is concentrated on $\{u_1, \cdots, u_m\}$ and $h_P(u_i)=h_{M}(u_i)$ for all $1\leq i\leq m$.  Let $\lambda=\widehat{V}_\phi(K,P)$. Then
\begin{eqnarray*}
1&=&\int_{S^{n-1}}\phi\left(\frac{n|K|\cdot h_P(u)}{\lambda \cdot h_K(u)}\right)\dV \\ &=& \frac{1}{n|K|} \cdot  \sum_{i=1}^m  \phi\left(\frac{n|K|\cdot h_P(u_i)}{\lambda \cdot h_K(u_i)}\right)h_K(u_i) S_K(u_i) \\ &=& \frac{1}{n|K|}\cdot  \sum_{i=1}^m  \phi\left(\frac{n|K|\cdot h_M(u_i)}{\lambda \cdot h_K(u_i)}\right)h_K(u_i) S_K(u_i) \\&=&\int_{S^{n-1}}\phi\left(\frac{n|K|\cdot h_{M}(u)}{\lambda\cdot h_K(u)}\right)\dV. \end{eqnarray*} Consequently, $\lambda=\widehat{V}_\phi(K, P)=\widehat{V}_\phi(K, M)$.

As $P$ circumscribes $M$, then $P^\circ\subset M^\circ$ and $|P^\circ|\leq |M^\circ |=\omega_n$ with equality if and only if $M=P$.  Formula (\ref{affine surface area hat}) and Corollary \ref{corollary:homogeneous-1} yield that  for $\phi\in \widehat{\Phi}_1$,
$$ \affineg(K)\leq \widehat{V}_\phi(K,\mathrm{vrad}(P^\circ)P)\leq \widehat{V}_\phi(K, M)=\affineg(K).$$  This requires in particular $|P^\circ|= |M^\circ |=\omega_n$. Hence $M=P$ is a polytope whose faces are parallel to those of $K$.
\end{proof}

\bp \label{l2} Let $K\in\mathcal{K}_{0}$ and $r_K ,R_K>0$ be such
that $r_K \ball\subset K\subset R_K\ball$. For $\phi\in\widehat{\Phi}_1$ and $M\in \widehat{T}_{\phi}K$, there exists an integer $j_0>1$ such
that, for all $u\in \sphere$,
\begin{eqnarray*}
h_M(u)\le
\frac{j_0\cdot R_K^{n+1}}{r_K^{n+1}}\cdot \phi^{-1}\left(\frac{2n\omega_n\cdot R_K^n}{c_1\cdot
r_K}\right),\end{eqnarray*} where $c_1>0$ is the constant in (\ref{sphere integration}).
\ep

\begin{proof} Let $M\in \widehat{T}_{\phi}K$. First of all, the minimality of $M$ gives that
$$\widehat{V}_\phi(K,M)\le \widehat{V}_\phi(K,\ball)\le
 \frac{n\omega_n\cdot R_K^n} {r_K},$$ where the second inequality follows from Lemma \ref{l-bounded}. Let $\lambda=\widehat{V}_\phi(K, M)$ and $R(M)=\rho_{M}(v)=\max\{\rho_{M}(u):u\in
S^{n-1}\}$. A calculation similar to (\ref{projection-estimate-0717}) leads to
\begin{eqnarray*}
1&=& \int_{S^{n-1}}\phi\left(\frac{n|K|\cdot
h_{M}(u)}{\lambda\cdot h_K(u)}\right)d\widetilde{V}_K(u)\\
&\ge& \int_{S^{n-1}}\phi\left(\frac{n|K|\cdot R(M)\cdot
\langle u, v\rangle_+}
{\lambda \cdot R_K}\right)\frac{r_K}{n|K|}dS_{K}(u)\\
&\ge& \int_{\Sigma_{j_0}(v)} \phi\left(\frac{n|K|\cdot R(M)}{\lambda\cdot j_0\cdot
R_K}\right) \frac{r_K}{n|K|}dS_{K}(u)\\
&\ge & \phi\left(\frac{n|K|\cdot R(M)}{\lambda\cdot j_0\cdot
R_K}\right)\frac{r_K \cdot c_1}{2n|K|}.
\end{eqnarray*} By the facts that $\phi(1)=1$ and $\phi$ is increasing, one has
$$
R(M)\le \frac{\lambda\cdot j_0\cdot
R_K}{n|K|}\cdot \phi^{-1}\left(\frac{2n|K|}{c_1 \cdot r_K}\right)\le
\frac{j_0\cdot R_K^{n+1}}{r_K^{n+1}}\cdot \phi^{-1}\left(\frac{2n\omega_n \cdot R_K^n}{c_1\cdot
r_K}\right). 
$$ This completes the proof. \end{proof}

\subsection{Continuity of the homogeneous Orlicz geominimal surface areas} \label{Continuity of the homogeneous Orlicz geominimal surface area}
This subsection is dedicated to prove the continuity of the
homogeneous Orlicz geominimal surface areas under the condition
$\phi\in \widehat{\Phi}_1$. The following uniform boundedness
argument is needed.

\bl \label{l3} Let $\{K_\alpha \}_{\alpha \in \Lambda}\subset\cK_0$ be a family of convex bodies satisfying the uniformly bounded property: there exist constants $r, R>0$  such that $ r\ball\subset K_\alpha\subset
R\ball$ for all $\alpha \in \Lambda$. For $\phi\in\widehat{\Phi}_1$ and for any $M_{\alpha}\in \widehat{T}_{\phi}(K_{\alpha})$, there exist constants $r',R'>0$ such
that $$r'\ball \subset M_{\alpha} \subset R'\ball  \ \ \ \ \mathrm{for \ all}\ \ \ \alpha\in \Lambda.$$
\el

\begin{proof} We only need to prove the case that $\{K_{\alpha}\}_{\alpha\in \Lambda}$ contains infinite many different convex bodies, as otherwise the argument is trivial. 

Let $M_\alpha\in \widehat{T}_{\phi}(K_{\alpha})$. First, we prove the existence of $R'$ by contradiction. To this end, we assume that  there is no constant $R'$ such that $M_{\alpha} \subset R'\ball$ for all $\alpha\in \Lambda$. In other words, there is a sequence of $\{M_{\alpha}\}_{\alpha\in \Lambda}$, denoted by $\{M_i\}_{i\geq 1}$, such that $R(M_i)\rightarrow \infty$. Hereafter, for all $i\geq 1$,  $$R(M_i)=\rho_{M_i}(u_i)=\max\{\rho_{M_i}(u):u\in S^{n-1}\}.$$ Similar to the proof of Proposition \ref{p3}, one can find a subsequence, which will not be relabeled, such that, $u_i\rightarrow v\in \sphere$ (due to the compactness of $\sphere$), $R(M_i)\rightarrow \infty$ and $K_i\rightarrow K$ (by the Blaschke selection
theorem due to the uniform boundedness of $\{K_{\alpha}\}_{\alpha\in \Lambda}$) as $i\rightarrow \infty$.

 It follows from Proposition \ref{p2} that
$\widehat{V}_\phi(K_i,\ball)\rightarrow \widehat{V}_\phi(K,\ball)$
as $i\rightarrow \infty$. This implies the boundedness of the
sequence $\{\widehat{V}_\phi(K_i,\ball)\}_{i\geq 1}$ and hence
$$\lambda_i=\frac{\widehat{V}_\phi(K_i,\ball)}{R(M_i)} \rightarrow
0\ \ \ \ \mathrm{as}\ \ \ \ i\rightarrow \infty.$$  Let
$\varepsilon>0$ be given. The triangle inequality yields the uniform
convergence of $\langle u,u_i \rangle_+\rightarrow \langle u,v
\rangle_+ $ on $\sphere$ as $i\rightarrow \infty$. Moreover, as
$K_i\rightarrow K$, one sees $r\ball \subset K\subset R\ball$ and
\be\nonumber 0\leq \frac{n|K_i|\cdot \langle u, u_i
\rangle_+}{(\lambda_i+\varepsilon)\cdot h_{K_i}(u)}\leq
\frac{n\omega_n\cdot R^n}{\varepsilon \cdot r} \  \ \ \
\mathrm{and}\ \ \ 0\leq \frac{n|K|\cdot \langle u, v
\rangle_+}{\varepsilon\cdot h_K(u)}\leq \frac{n\omega_n\cdot
R^n}{\varepsilon \cdot r}. \ \
 \ee A simple calculation yields that
\begin{eqnarray*}
 \frac{n|K_i|\cdot \langle u,u_i \rangle_+}{(\lambda_i+\varepsilon)
 \cdot h_{K_i}(u)}\rightarrow \frac{n|K|\cdot \langle u,v\rangle_+}
 {\varepsilon\cdot h_K(u)}  \ \ \ \ \ \mathrm{uniformly \ on} \ \ \sphere
  \ \ \mathrm{as}\ \ i\rightarrow \infty.
\end{eqnarray*}
Let $I=[0, n\omega_n R^n \varepsilon^{-1} r^{-1}]$ and then $\phi\in \widehat{\Phi}_1$ is uniformly continuous on $I$. By Lemma \ref{uniform-continuous-convergence}, one gets
\begin{equation}\label{uniformly convergence}
\phi\left(\frac{n|K_i|\cdot \langle u,u_i \rangle_+}{(\lambda_i+\varepsilon)\cdot h_{K_i}(u)}\right)\rightarrow \phi\left(\frac{n|K|\cdot \langle u,v\rangle_+}{\varepsilon\cdot h_K(u)}\right) \ \ \ \mbox{uniformly \ on}\ \ \sphere \ \ \ \mbox{as}\  i\rightarrow \infty.
\end{equation}  Note that $\phi\in \widehat{\Phi}_1$ is increasing. By Corollary
\ref{corollary:homogeneous-2},  (\ref{weak fact}), (\ref{sphere integration}) and
(\ref{uniformly convergence}), a calculation similar to (\ref{projection-estimate-0717}) leads to, for any given $\varepsilon>0$,
\begin{eqnarray*}
1&=& \lim_{i\rightarrow \infty}\int_{S^{n-1}}\phi\left(\frac{n|K_i|\cdot h_{M_i}(u)}{\widehat{V}_\phi(K_i, M_i)\cdot h_{K_i}(u)}\right)d\widetilde{V}_{K_i}(u)\\
&\geq& \lim_{i\rightarrow \infty}\int_{S^{n-1}}\phi\left(\frac{n|K_i|\cdot R(M_i) \cdot \langle u,u_i\rangle_+}{\widehat{V}_\phi(K_i,\ball)\cdot h_{K_i}(u)}\right)d\widetilde{V}_{K_i}(u)\\
&\geq&
 \lim_{i\rightarrow \infty}\int_{S^{n-1}}  \phi\left(\frac{n|K_i|\cdot \langle u,u_i\rangle_+}{(\lambda_i+\varepsilon) \cdot h_{K_i}(u)}\right)d\widetilde{V}_{K_i}(u)\\
&=&\int_{S^{n-1}}\phi\left(\frac{n|K| \cdot \langle u,v\rangle_+}{\varepsilon\cdot h_{K}(u)}\right)\dV\\
 &=& G_v(\varepsilon).
\end{eqnarray*}
It follows from Lemma \ref{projection-positive-0715} that $\lim_{\varepsilon\rightarrow 0^+}G_v(\varepsilon)=\infty$, which leads to  a contradiction (i.e., $1\geq \infty$). Thus $R(M_i)\rightarrow \infty$ is impossible. This concludes the existence of $R'$ such that $M_{\alpha} \subset R'\ball$ for  all $\alpha\in \Lambda.$ In other words,  $\{M_{\alpha}\}_{\alpha\in\Lambda}\subset \cK_0$ is uniformly bounded.

\indent Next, we show the existence of $r'>0$ such that $r'\ball \subset M_{\alpha}$ for  all $\alpha\in \Lambda.$ Assume that there is no such a constant $r'>0$.  In other words,  there is a sequence $\{M_j\}_{j\geq 1}$ such that  $w_j\rightarrow w\in \sphere$  (due to the compactness of $\sphere$) and  $r_j\rightarrow 0$ as $j\rightarrow \infty$, where
 $$r_j=h_{M_j}(w_j)=\min\{h_{M_j}(u): u\in S^{n-1}\}. $$ Note that the sequence  $\{M_j\}_{j\geq 1}\subset \cK_0$ is uniformly
 bounded (as proved above). The Blaschke selection theorem, Lemma \ref{l1} and $|M_j^\circ|=\omega_n$ for all $j\geq 1$ imply that there exists a subsequence of
 $\{M_j\}_{j\geq 1}$, which will not be relabeled, and a convex body $M\in \cK_0$, such that, $
 M_j \rightarrow M$ as $j\rightarrow \infty$.    That is,
 $$ \lim_{j\rightarrow\infty}\sup_{u\in \sphere} |h_{M_j}(u)-h_M(u)|=0.$$ This further implies, as $w_j\rightarrow w$,   $$h_M(w)=\lim_{j\rightarrow\infty}h_{M_j}(w_j)=\lim_{j\rightarrow\infty}r_j=0.$$ This contradicts with the positivity of the support function of $M$. Hence, there is a constant  $r'>0$ such that $r'\ball \subset M_{\alpha}$ for  all $\alpha\in \Lambda$.
\end{proof}

Now let us prove our main result which states that the homogeneous Orlicz geominimal surface areas are continuous on $\cK_0$ with respect to the Hausdorff distance.

\bt\label{continuity-0717-1}  For $\phi\in\widehat{\Phi}_1$, the functional $\widehat{G}^{orilcz}_\phi(\cdot)$ on $\cK_0$ is continuous with respect to the Hausdorff distance. In particular, the $L_p$ geominimal surface surface area for $p\in (0, \infty)$ is continuous on $\cK_0$ with respect to the Hausdorff distance. \et

\bpf The upper semicontinuity has been proved in Proposition \ref{semicontinuous-07-17-1----1}. To get the continuity, it is enough to prove that the homogeneous Orlicz geominimal surface areas are lower
semicontinuous on $\mathcal{K}_0$.  To this end, let
$\{K_i\}_{i\geq 1} \subset \cK_0$ be a convergent sequence whose limit is $K_0\in\cK_0$. Let $M_i\in \widehat{T}_{\phi}(K_i)$ for $i\geq 1$. Clearly, $\{K_i\}_{i\geq 0}$ satisfies the uniformly bounded condition in Lemma  \ref{l3}, which implies the uniform boundedness of the sequence
$\{M_i\}_{i\geq 1}$.

Let $l=\liminf_{i\rightarrow \infty} \affineg(K_i).$ Consequently, one can find a subsequence $\{K_{i_k}\}_{k\geq 1}$ such that $l=\lim_{k\rightarrow \infty} \affineg(K_{i_k}).$
By the Blaschke selection
theorem and Lemma \ref{l1}, there exists a subsequence of $\{M_{i_k}\}_{k\geq 1}$ (still
denoted by $\{M_{i_k}\}_{k\geq 1}$) and a body
$M\in\cK_0$, such that,  $M_{i_k}\rightarrow M$ as $k\rightarrow \infty$ and
$|M^\circ|=\omega_n$. Proposition \ref{p2} then yields $$
\widehat{G}^{orlicz}_\phi(K_{i_k})=\widehat{V}_\phi(K_{i_k}, M_{i_k})\rightarrow \widehat{V}_\phi(K_0, M) \ \ \ \mathrm{as} \ \ \ k\rightarrow \infty.$$ It follows from (\ref{equi-formula-geominimal-0715}) that   $$
\widehat{G}^{orlicz}_\phi(K_0)\leq \widehat{V}_\phi(K_0, M)=\lim_{k\rightarrow\infty} \widehat{G}^{orlicz}_\phi(K_{i_k})=\liminf_{i\rightarrow \infty} \widehat{G}^{orlicz}_\phi(K_i).
$$
This completes the proof. \epf

Proposition \ref{p3} states that if $\phi\in\widehat{\Phi}_1$ is convex, the $L_{\phi}$ Orlicz-Petty body is unique. In this case, $\widehat{T}_\phi K$ contains only one element. Consequently, $\widehat{T}_{\phi}: \cK_0\rightarrow \cK_0$ defines an operator.  The following result states that the operator $\widehat{T}_{\phi}$ is continuous.

\bp \label{continuity of petty bodies}  Let $\phi\in\widehat{\Phi}_1$ be convex. Then $\widehat{T}_\phi:
 \cK_0\mapsto \cK_0$ is continuous with respect to the Hausdorff distance.\ep

 \bpf It is enough to prove that $\{\widehat{T}_{\phi}K_i\}_{i\geq 1}\subset \cK_0$ is convergent to $\widehat{T}_{\phi}K_0\in \cK_0$ for every convergent sequence $\{K_i\}_{i\geq 1}\subset \cK_0$ with limit $K_0\in \cK_0$,  in particular, every subsequence of $\{\widehat{T}_\phi K_i\}_{i\geq 1}$ has a convergent subsequence whose limit is $\widehat{T}_{\phi}K_0$.

 Let $\{K_{i_k}\}_{k\geq 1}$ be any subsequence of $\{K_i\}_{i\geq 1}$. Of course, $K_{i_k}\rightarrow K_0$ as $k\rightarrow \infty$ and $\{\widehat{T}_{\phi} K_{i_k}\}_{k\geq 1}$ is uniformly bounded by Lemma \ref{l3}.  Following the Blaschke selection theorem, one can find
 a subsequence of $\{\widehat{T}_\phi K_{i_k}\}_{k\geq 1}$, which will not be relabeled,  and $M\in\cK_0$ such that $\widehat{T}_\phi K_{i_k}\rightarrow M$ as $k\rightarrow \infty$  and
 $|M^\circ|=\omega_n$. By Proposition \ref{p2}, one has
 $$
 \widehat{G}^{orlicz}_\phi(K_{i_k})=\widehat{V}_\phi(K_{i_k}, \widehat{T}_\phi K_{i_k})\rightarrow
 \widehat{V}_\phi(K_0,M) \ \ \ \mathrm{as}\ \ \ k\rightarrow \infty. $$
By Theorem \ref{continuity-0717-1}, one has
 $$ \widehat{G}^{orlicz}_\phi(K_{i_k})\rightarrow \widehat{G}^{orlicz}_\phi(K_0)
 =\widehat{V}_\phi(K_0, \widehat{T}_\phi K_0) \ \ \ \mathrm{as}\ \ \ k\rightarrow \infty.$$
Hence, $\widehat{V}_\phi(K_0, \widehat{T}_\phi
K_0)=\widehat{V}_\phi(K_0,M)$ and then $\widehat{T}_\phi K_0=M$ by the
uniqueness of the $L_\phi$ Orlicz-Petty body for $\phi\in \widehat{\Phi}_1$ being convex.
 \epf

\section{The nonhomogeneous Orlicz geominimal surface areas}\label{section-nonhomogeneous-gemo}

 In this section, we will briefly discuss the continuity of the
 nonhomogeneous Orlicz geominimal surface areas defined 
 in \cite{Ye2015b}. In particular, we prove the existence, uniqueness
 and affine invariance  for the $L_\varphi$ Orlicz-Petty bodies in Subsection
  \ref{section-nonhomogeneous-petty}. In Subsection \ref{section-geom-inter}, we provide  a 
  geometric interpretation for the nonhomogeneous Orlicz $L_{\varphi}$ mixed
  volume with $\varphi\in \mathcal{I}\cup\mathcal{D}$ (in particular, for $\varphi(t)=t^p$ with $p<1$).

\subsection{The geometric interpretation for the nonhomogeneous Orlicz $L_{\varphi}$ mixed volume}\label{section-geom-inter}

 For any continuous function $\varphi: (0, \infty)\rightarrow (0, \infty)$, $V_{\varphi}(K, L)$ denotes the nonhomogeneous Orlicz $L_{\varphi}$ mixed volume of $K$ and $L$. It has the following integral expression:  \begin{equation}\label{definition-mixed-volume-2016-07-17-1} V_{\varphi}(K, L)= \frac{1}{n} \int_{S^{n-1}}
\varphi\bigg(\frac{h_L(u)}{h_K(u)}\bigg)h_K(u)dS_K(u).\end{equation}
  We can use the following examples to see that $V_{\varphi}(\cdot, \cdot)$ is
not homogeneous:$$V_{\varphi}(r\ball, \ball)=\varphi(1/r)\cdot
r^n\cdot\omega_n \ \ \ \mathrm{and} \ \ \ V_{\varphi}(\ball,
r\ball)=\varphi(r)\cdot \omega_n.$$ The geometric interpretation of
$V_{\varphi}(\cdot, \cdot)$ for  convex  $\varphi\in \mathcal{I}$
was given in \cite{Gardner2014, XJL}. However, there are no
geometric interpretations of $V_{\varphi}(\cdot, \cdot)$ for
non-convex functions $\varphi$ (even if $\varphi(t)=t^p$ for $p<1$).
In this subsection, we will provide such a geometric interpretation
for all $\varphi\in \mathcal{I}\cup\mathcal{D}$.

 Denote by $C^+(S^{n-1})$ the set of all positive continuous functions on
$S^{n-1}$. Define $K_f$, the Aleksandrov body associated with $f\in
C^+(S^{n-1})$, by
$$K_f=\cap_{u\in S^{n-1}} H^-(u,f(u)),$$
where $H^-(u,\alpha)$ is the half space with normal vector $u$ and constant $\alpha>0$:
$$H^-(u,\alpha)=\{x\in\mathbb{R}^n:  \langle x,u\rangle\le \alpha\}.$$ This implies that  $$
K_f=\{x\in\mathbb{R}^n:   \langle x,u\rangle\le f(u) \quad\text{for
all}\quad u\in S^{n-1}\}.
$$  Equivalently, $K_f$ is the (unique) maximal element (with
respect to set inclusion) of the set $$
\{K\in\mathcal{K}_0: h_K(u)\le f(u) \quad\text{for all}\quad u\in
S^{n-1}\}. $$  When $f=h_L$ for some convex body $L\in \cK_0$, one sees $K_f=L$.

 For $K\in \mathcal{K}_0$ and
$f\in C^+(S^{n-1})$, the $L_1$ mixed volume of $K$ and $f$, denoted by $V_1(K, f)$, can be formulated by
$$V_1(K, f)=\frac{1}{n}\int_{S^{n-1}}f(u)dS_K(u).$$ When $f$ is the support function of a convex body $L$, then $V_1(K, f)$ is just the usual $L_1$ mixed volume of $K$ and $L$ (i.e., $\varphi(t)=t$ in formula (\ref{definition-mixed-volume-2016-07-17-1})). In particular, $V_1(K, h_K)=|K|$ for all $K\in \cK_0$. Lemma 3.1 in \cite{Lut1993} states that \begin{eqnarray} |K_f|=V_1(K_f, f). \label{Alexander-0718-1} \end{eqnarray}

In order to prove the geometric interpretation for $V_{\varphi}(\cdot, \cdot)$, the linear Orlicz addition of functions \cite{HouYe2016} is needed.
A special case is given below. 
\bd Assume that either $\varphi_1, \varphi_2\in \mathcal{I}$ or $\varphi_1, \varphi_2\in \mathcal{D}$. For $\varepsilon>0$, define $p_1+_{\varphi,\varepsilon} p_2$, the linear Orlicz addition of positive functions $p_1, p_2$ (on whatever common domain),  by
\begin{equation*}
\varphi_1\left(\frac{p_1(x)}{(p_1+_{\varphi,\varepsilon}
p_2)(x)}\right)+\varepsilon\varphi_2
\left(\frac{p_2(x)}{(p_1+_{\varphi,\varepsilon} p_2)(x)}\right)=1.
\end{equation*} \ed

For our context, $p_1=h_K$ and $p_2=h_L$ where $K, L\in\cK_0$ are
two convex bodies. Namely we let $f_{\varepsilon}=h_K+_{\varphi,
\varepsilon} h_L$ and then for any $u\in\sphere$, \be \label{orlicz
equation}
\varphi_1\left(\frac{h_K(u)}{f_{\varepsilon}(u)}\right)+\varepsilon\varphi_2\left(\frac{h_L(u)}{f_{\varepsilon}(u)}\right)=1.
\ee When $\varphi_1,\varphi_2\in \mathcal{I}$ are convex functions,
$f_\varepsilon=h_K+_{\varphi, \varepsilon} h_L$ is the support
function of a convex body (see \cite{Gardner2014, XJL}). Clearly,
$f_\varepsilon\in C^+(\sphere)$ determines an Aleksandrov body
$K_{f_\varepsilon}$, which will be written as  $K_\varepsilon$  for
simplicity.  Moreover,   $h_K\leq f_\varepsilon$ if
$\varphi_1,\varphi_2\in \mathcal{I}$ and $h_K\geq f_\varepsilon$ if
$\varphi_1,\varphi_2\in \mathcal{D}$.

Let $(\varphi_1)'_l(1)$ and $(\varphi_1)'_r(1)$ stand for the left
and the right derivatives of $\varphi_1$ at $t=1$, respectively, if
they exist. From the proof of Theorem 9 in \cite{HouYe2016}, one
sees that $f_{\varepsilon}\rightarrow h_K$ uniformly on $\sphere$ as
$\varepsilon\rightarrow 0^+$.  Following similar arguments in
\cite{Gardner2014, ghwy14, HouYe2016, Zhu2015}, we can prove the following
result.

\bl \label{convergence 2-0718-1} Let $K, L\in\mathcal{K}_0$ and $\varphi_1,\varphi_2\in \mathcal{I}$ be such that $(\varphi_1)'_l(1)$ exists and is positive. Then \begin{eqnarray}
{(\varphi_1)'_l(1)} \lim_{\varepsilon\rightarrow
0^+}\frac{f_\varepsilon(u)-h_K(u)}{\varepsilon}
&=&{h_K(u)}\cdot \varphi_2\left(\frac{h_L(u)}{h_K(u)}\right) \ \ \ \mathrm{uniformly\ on}\ \  S^{n-1}. \label{convergence of f-e}
\end{eqnarray}
 For $\varphi_1,\varphi_2\in \mathcal{D}$,  (\ref{convergence of f-e}) holds with $(\varphi_1)'_l(1)$ replaced by $(\varphi_1)'_r(1)$. \el

 \begin{proof} Let $\varphi_1,\varphi_2\in \mathcal{I}$. Note that $f_{\varepsilon}\downarrow h_K$ uniformly on $\sphere$  as $\varepsilon\downarrow 0^+$. Then, for all $u\in \sphere$,
\begin{eqnarray*}
(\varphi_1)_{l}^{'}(1) &=& \lim_{\varepsilon\rightarrow 0^+}f_\varepsilon(u) \cdot \frac{1-\varphi_1\big(\frac{h_K(u)}{f_\varepsilon(u)}\big)}{f_\varepsilon(u)-h_K(u)}\\
&=& \lim_{\varepsilon\rightarrow 0^+}f_\varepsilon(u)\cdot \varphi_2\bigg(\frac{h_L(u)}{f_\varepsilon(u)}\bigg)\cdot \frac{\varepsilon}{f_\varepsilon(u)-h_K(u)}\\&=& h_K(u)\cdot \varphi_2\left(\frac{h_L(u)}{h_K(u)}\right)\cdot \lim_{\varepsilon\rightarrow0^+}\frac{\varepsilon}{f_\varepsilon(u)-h_K(u)}.
\end{eqnarray*}
Rewrite the above limit as follows: \begin{eqnarray*}{(\varphi_1)'_{l}(1)} \cdot \lim_{\varepsilon\rightarrow0^+}\frac{f_\varepsilon(u)-h_K(u)}{\varepsilon}=  \lim_{\varepsilon\rightarrow 0^+}f_\varepsilon(u)\cdot  \lim_{\varepsilon\rightarrow 0^+} \varphi_2\bigg(\frac{h_L(u)}{f_\varepsilon(u)}\bigg)= h_K(u)\cdot \varphi_2\left(\frac{h_L(u)}{h_K(u)}\right).
\end{eqnarray*} Moreover, the convergence is uniform because both  $\{f_\varepsilon(u)\}_{\varepsilon>0}$  and  $\big\{\varphi_2\big(\frac{h_L(u)}{f_\varepsilon(u)}\big)\big\}_{\varepsilon>0}$  are uniformly convergent and uniformly bounded on $\sphere$.

If $\varphi_1,\varphi_2\in\mathcal{D}$ such that $(\varphi_1)_{r}^{'}(1)$ exists and is nonzero, the proof goes along the same manner.
\end{proof}

The geometric interpretation for the nonhomogeneous Orlicz $L_{\varphi}$
mixed volume with $\varphi\in \mathcal{I}\cup\mathcal{D}$ is given in the following theorem.  The following result by  Aleksandrov \cite{Ale1938-1} is needed: if the sequence $\{f_i\}_{i\geq 1} \subset C^{+}(S^{n-1})$ converges to $f\in C^{+}(S^{n-1})$ uniformly on $\sphere$, then the sequence of $\{K_{f_i}\}_{i\geq 1}$, the Aleksandrov bodies associated to $f_i$, converges to $K_f$ with respect to the Hausdorff distance.

 \bt Let $K, L\in\mathcal{K}_0$ and $\varphi_1,\varphi_2\in \mathcal{I}$ be such that $(\varphi_1)'_l(1)$ exists and is positive. Then,
 \be
V_{\varphi_2}(K,L)=\frac{(\varphi_1)'_l(1)}{n}\lim_{\varepsilon\rightarrow
0^+}\frac{|K_{\varepsilon}|-|K|}{\varepsilon}.\label{convergence of f-e---1}
 \ee
 For $\varphi_1,\varphi_2\in \mathcal{D}$,  (\ref{convergence of f-e---1}) holds with $(\varphi_1)'_l(1)$ replaced by $(\varphi_1)'_r(1)$.
  \et
\begin{proof} The uniform convergence of $f_\varepsilon$ on $\sphere$ implies that $K_{\varepsilon}$  converges to $K$ in the Hausdorff distance as $\varepsilon\rightarrow 0^+$. In particular $|K_{\varepsilon}|\rightarrow |K|$ as $\varepsilon\rightarrow 0^+$ and $S_{K_\varepsilon}$ converges to $S_K$ weakly  on $\sphere$.  It follows from  (\ref{weak fact}),   (\ref{Alexander-0718-1}), the Minkowski inequality (\ref{classical Minkowski inequality}) and  Lemma \ref{convergence 2-0718-1}  that  \begin{eqnarray*} \liminf_{\varepsilon\rightarrow 0^+}|K_\varepsilon|^\frac{n-1}{n}\cdot
\frac{|K_\varepsilon|^\frac{1}{n}-|K|^\frac{1}{n}}{\varepsilon}
&\ge&\liminf_{\varepsilon\rightarrow
0^+}\frac{|K_\varepsilon|-V_1(K_\varepsilon, K)}{\varepsilon}\\
&=&\liminf_{\varepsilon\rightarrow 0^+}\frac{V_1(K_\varepsilon,
f_\varepsilon)-V_1(K_\varepsilon,
h_{K})}{\varepsilon}\\&=&\lim_{\varepsilon\rightarrow
0^+}\frac{1}{n}\int_{S^{n-1}}\frac{f_\varepsilon(u)
-h_K(u)}{\varepsilon}dS_{K_\varepsilon}(u)
 \\
&=&\frac{1}{(\varphi_1)'_l(1)}V_{\varphi_2}(K,L).
 \end{eqnarray*} Similarly, due to  $h_{K_\varepsilon}\le f_\varepsilon$,   \begin{eqnarray*}|K|^\frac{n-1}{n}\cdot  \limsup_{\varepsilon\rightarrow 0^+}
\frac{|K_\varepsilon|^\frac{1}{n}-|K|^\frac{1}{n}}{\varepsilon}
&\le&\limsup_{\varepsilon\rightarrow 0^+}\frac{V_1(K,
K_\varepsilon)-|K|}{\varepsilon}\\
&\le& \limsup_{\varepsilon\rightarrow 0^+}\frac{V_1(K,
f_\varepsilon)-V_1(K, h_{K})}{\varepsilon}\\
&=&\lim_{\varepsilon\rightarrow
0^+}\frac{1}{n}\int_{S^{n-1}}\frac{f_\varepsilon(u)
-h_K(u)}{\varepsilon}dS_K(u) \\ &=& \frac{1}{(\varphi_1)'_l(1)}V_{\varphi_2}(K,L).
\end{eqnarray*} Combing the inequalities above,  one has
$$(\varphi_1)'_l(1) \cdot |K|^\frac{n-1}{n} \cdot  \lim_{\varepsilon\rightarrow 0^+}
\frac{|K_\varepsilon|^\frac{1}{n}-|K|^\frac{1}{n}}{\varepsilon}=
V_{\varphi_2}(K, L). $$
Let $g(\varepsilon)=|K_\varepsilon|^\frac{1}{n}$ and $g(0)=|K|^\frac{1}{n}$. Then  \begin{eqnarray*}
\frac{(\varphi_1)'_l(1)}{n}\cdot \lim_{\varepsilon\rightarrow 0^+}
\frac{|K_\varepsilon|-|K|}{\varepsilon}
&=&\frac{(\varphi_1)'_l(1)}{n}\cdot  \lim_{\varepsilon\rightarrow 0^+}\frac{g(\varepsilon)^n-g(0)^n}{\varepsilon}\\
&=& (\varphi_1)'_l(1) \cdot g(0)^{n-1}\lim_{\varepsilon\rightarrow 0^+}\frac{g(\varepsilon)-g(0)}{\varepsilon}\\
&=&V_{\varphi_2}(K,L).
\end{eqnarray*}
The result for $\varphi_1, \varphi_2\in \mathcal{D}$ follows along
the same lines.  \end{proof}

Let $\varphi_1(t)=\varphi_2(t)=t^p$ for $0\neq p\in \bbR$. Then formula (\ref{orlicz equation}) gives the $L_p$ addition of $h_K$ and $h_L$:
$$f_{p, \varepsilon}(u) =\big[h_K(u)^p+\varepsilon h_L(u)^p\big]^{1/p}\ \ \ \ \mathrm{for}\ \ u\in \sphere.$$ Then  the $L_p$ mixed volume of $K$ and $L$ \cite{Lut1993, Ye2015a} is the first order variation at $\varepsilon=0$ of the volume of $K_{f_{p, \varepsilon}}$, the Aleksandrov body associated to $f_{p, \varepsilon}$: \be
V_p(K,L)=\frac{p}{n} \cdot \lim_{\varepsilon\rightarrow
0^+}\frac{|K_{f_{p, \varepsilon}}|-|K|}{\varepsilon}.\nonumber
 \ee

\subsection{The Orlicz-Petty bodies and the continuity of nonhomogeneous Orlicz geominimal surface areas}\label{section-nonhomogeneous-petty}

In this subsection, we establish the continuity of the nonhomogeneous
Orlicz geominimal surface areas, whose proof is similar to that
in Section \ref{section-homogeneous-petty}. For completeness, we
still include the proof with emphasis on the modification.

The nonhomogeneous Orlicz geominimal surface areas can be defined as
follows. \bd \label{definition-0719} Let $K\in\mathcal{K}_{0}$ be a
convex body
with the origin in its interior.\\
(i) For $\varphi\in\widehat{\Phi}_1\cup\widehat{\Psi}$, define the nonhomogeneous
Orlicz $L_\varphi$ geominimal surface area of $K$ by
\begin{eqnarray*}
G^{orlicz}_\varphi(K)=\inf_{L\in\cK_0}\{nV_\varphi(K,\mathrm{vrad}(L^\circ)L)\}=\inf\{nV_\varphi(K,L):L\in\mathcal{K}_{0}
\ \mathrm{with}\ \  |L^\circ|=\omega_n\}.
\end{eqnarray*}
(ii) For $\varphi\in\widehat{\Phi}_2$, define the nonhomogeneous Orlicz $L_\varphi$
geominimal surface area of $K$ by
\begin{eqnarray*}
G^{orlicz}_\varphi(K)=\sup_{L\in\cK_0}\{nV_\varphi(K,\mathrm{vrad}(L^\circ)L)\}=\sup\{nV_\varphi(K,L):L\in\mathcal{K}_{0}
\ \mathrm{with}\ \ |L^\circ|=\omega_n\}.
\end{eqnarray*}
\ed Note that the nonhomogeneous Orlicz $L_\varphi$ geominimal
surface area can be defined for more general functions than
$\varphi\in \mathcal{I}\cup\mathcal{D}$ (see more details in \cite{Ye2015b}). However, from Section \ref{section-homogeneous-petty},
one sees that the monotonicity of $\varphi$ is crucial to
establish continuity of Orlicz geominimal surface areas. Hence, in
this section, we only consider
$\varphi\in\widehat{\Phi}\cup\widehat{\Psi}$. We can use the
following example to see that $G^{orlicz}_\varphi(\cdot)$ is not
homogeneous (see Corollary 3.1 in
\cite{Ye2015b}):$$G^{orlicz}_\varphi(r\ball)=\varphi(1/r)\cdot
r^n\cdot n\omega_n.$$

\bp\label{p22--222} Let $\{K_i\}_{i\geq 1}\subset \cK_0$ and $\{L_i\}_{i\geq 1}\subset \cK_0$
be such that
$K_i\rightarrow K\in\mathcal{K}_{0}$ and $L_i\rightarrow
L\in\mathcal{K}_{0}$ as $i\rightarrow \infty$. For
$\varphi\in\widehat{\Phi}\cup \widehat{\Psi}$,  one has
$V_\varphi(K_i,L_i)\rightarrow V_\varphi(K, L)$ as $i\rightarrow
\infty$. \ep

\begin{proof} As $K_i\rightarrow K\in\mathcal{K}_{0}$ and
$L_i\rightarrow L\in\mathcal{K}_{0}$,  one can find constants $r,
R>0$ such that these bodies contain $r\ball$ and are
contained in $R\ball$. Moreover, $h_{K_i}\rightarrow h_{K}$
and $h_{L_i}\rightarrow h_{L}$ uniformly on $S^{n-1}$. Together with
Lemma \ref{uniform-continuous-convergence} (where we can let
$I=[r/R, R/r]$),   one has
\begin{eqnarray*}
\varphi\left(\frac{h_{L_i}(u)}{h_{K_i}(u)}\right)h_{K_i}(u)\rightarrow
\varphi\left(\frac{h_L(u)}{h_K(u)}\right)h_K(u)  \quad
\mathrm{uniformly\ on} \quad S^{n-1}.
\end{eqnarray*} Formula (\ref{weak fact}) then implies \begin{eqnarray*}
\int_{S^{n-1}}\varphi\left(\frac{h_{L_i}(u)}{h_{K_i}(u)}\right)h_{K_i}(u)dS_{K_i}(u)\rightarrow
\int_{S^{n-1}}\varphi\left(\frac{h_L(u)}{h_K(u)}\right)h_K(u)dS_K(u).
\end{eqnarray*}
This completes the proof. \end{proof}

Similar to Proposition \ref{semicontinuous-07-17-1----1},  the
nonhomogeneous Orlicz $L_{\varphi}$ geominimal surface area is upper
(lower, respectively) semicontinuous on $\cK_0$ with respect to the
Hausdorff distance for $\varphi\in
\widehat{\Phi}_1\cup\widehat{\Psi}$ (for
$\varphi\in\widehat{\Phi}_2$, respectively).

The following proposition states that the Orlicz-Petty bodies exist.
See \cite{YJL} for special results when $\varphi\in \mathcal{I}$ is
convex (in this case, $\varphi\in\widehat{\Phi}_1$).

\bp \label{p33} Let $K\in\mathcal{K}_{0}$ and  $\varphi\in
\widehat{\Phi}_1 $. There exists a convex body
$M\in\mathcal{K}_{0}$ such that
\begin{eqnarray*}
G^{orlicz}_\varphi(K)=nV_\varphi(K, M)  \quad  \mathrm{and}  \quad
|M^\circ|=\omega_n.
\end{eqnarray*}
If in addition $\varphi$ is convex, such a convex
body is unique. \ep

\begin{proof} Let $\varphi\in \widehat{\Phi}_1$. It follows from the definition of $G^{orlicz}_\varphi(K)$ that
there exists a sequence $\{M_i\}_{i\geq 1}\subset \cK_{0}$ such that
$nV_\varphi(K,M_i)  \rightarrow G^{orlicz}_\varphi(K)$,
$|M_i^\circ|=\omega_n$ and $2V_\varphi(K,\ball)\ge V_\varphi(K,M_i)$
for all  $i\geq 1$.  Let $R_i=\rho_{M_i}(u_i)=\max\{\rho_{M_i}(u):
u\in S^{n-1}\}$ and assume that  $\sup_{i\geq 1} R_i=\infty.$
Without loss of generality, let $R_i\rightarrow \infty$ and
$u_i\rightarrow v$ (due to the compactness of $\sphere$) as
$i\rightarrow \infty$. As before, $h_{M_i}(u)\geq R_i\cdot\langle
u,u_i\rangle_+$ for all $u\in S^{n-1}$.

Let  $D>0$ be given.  By Definition \ref{definition-0719}, Fatou's
lemma, continuity of $\varphi$, (\ref{sphere integration}) and the
fact that $\varphi$ is increasing, one has
\begin{eqnarray*}
2V_\varphi(K,\ball)&\ge&  \lim_{i\rightarrow \infty}\frac{1}{n}\int_{S^{n-1}}\varphi\left(\frac{h_{M_i}(u)}{h_K(u)}\right)h_K(u)dS_K(u)\\
&\ge& \liminf_{i\rightarrow \infty}\frac{1}{n}\int_{S^{n-1}}\varphi\left(\frac{R_i\cdot\langle u,u_i\rangle_+}{R_K}\right)r_K dS_K(u)\\&\ge& \liminf_{i\rightarrow \infty}\frac{1}{n}\int_{S^{n-1}}\varphi\left(\frac{D\cdot \langle u,u_i\rangle_+}{R_K}\right)r_K dS_K(u)\\
&\geq& \frac{1}{n}\int_{S^{n-1}}\varphi\left(\frac{D\cdot \langle u,v\rangle_+}{R_K}\right)r_K dS_K(u)\\
&\geq&  \varphi\left(\frac{D}{j_0\cdot R_K}\right)\cdot
\frac{r_K}{n}\cdot \frac{c_1}{2}.
\end{eqnarray*} A contradiction (i.e.,  $2V_\varphi(K,\ball)>\infty$) is
obtained by letting $D\rightarrow \infty$ and the fact that $\lim_{t\rightarrow\infty}\phi(t)=\infty$ (as
$\varphi\in \widehat{\Phi}_1$ is increasing and unbounded).  That is,
$\{M_i\}_{i\geq 1} $ is uniformly bounded, and a convergent
subsequence of $\{M_i\}_{i\geq 1}$, which will not be relabeled,  can
be found due to the Blaschke selection theorem.  Let $M$ be the
limit of $\{M_i\}_{i\geq 1}$ and then $M\in \cK_0$  due to Lemma
\ref{l1}. Moreover,  $|M_i^\circ|=\omega_n$ for all $i\geq 1$
implies $|M^\circ|=\omega_n$. It follows from Proposition
\ref{p22--222} that $M$ is the desired body such that
$G^{orlicz}_\varphi(K)=nV_\varphi(K, M)$ and $ |M^\circ|=\omega_n.$

For uniqueness, let $M_1,M_2\in \mathcal{K}_0$  be such that
$|M_1^\circ|=|M_2^\circ|=\omega_n$ and $${V}_\varphi(K,
M_1)=\inf_{L\in\cK_0}\{V_\varphi(K,
\vrad(L^\circ)L)\}={V}_\varphi(K, M_2).$$ Let
$M=\frac{M_1+M_2}{2}$. Then $\mathrm{vrad}(M^\circ)\leq 1$
with equality if and only if $M_1=M_2$ (see inequality
(\ref{mean-body-volume})). The fact that $\varphi$ is convex yields
that $V_\varphi(K,M)\leq V_\varphi(K,M_1)$. Therefore, if $M_1\neq
M_2$ (hence $\vrad(M^\circ)<1$), the fact that $\varphi$ is strictly increasing implies that
$$nV_\varphi(K,\vrad(M^\circ) M)<nV_\varphi(K,M)\leq
nV_\varphi(K,M_1)=nV_\varphi(K, \vrad(M_1^\circ)M_1).$$ This
contradicts with the minimality of $M_1$ and hence the uniqueness
follows.
\end{proof}

\bd Let $K\in\mathcal{K}_{0}$ and  $\varphi\in \widehat{\Phi}_1 $. A convex 
body $M\in \cK_0$ is said to be an $L_{\varphi}$ Orlicz-Petty body
of $K$, if $M\in \cK_0$ satisfies   \be \nonumber
G^{orlicz}_\varphi(K)=nV_\varphi(K, M)  \quad  \mathrm{and} \quad
|M^\circ|=\omega_n. \ee Denote by ${T}_\varphi K$ the set of all
$L_{\varphi}$ Orlicz-Petty bodies of $K$.   \ed

Let $\varphi\in \widehat{\Phi}_1$. The set ${T}_\varphi K$ has many
 properties same as those for $\widehat{T}_\phi K$. For instance,
 ${T}_{\varphi} K$ is $SL(n)$-invariant: $ T_\varphi (AK)=
A(T_\varphi K)$  for all $A\in SL(n)$. Moreover, if $K$ is a
polytope, then any convex body in $T_\varphi K$  must be a polytope
with faces parallel to those of $K$.  If in addition $\varphi$ is
convex,  $|T_{\varphi}K|\cdot |(T_{\varphi}K)^\circ| \leq |K|\cdot |K^\circ|.$

The continuity of the nonhomogeneous Orlicz $L_{\varphi}$ geominimal
surface areas is proved in the following theorem.  See \cite{YJL} for
special results when $\varphi\in \mathcal{I}$ is convex (in this case, $\varphi\in\widehat{\Phi}_1$).

\bt If $\varphi\in \widehat{\Phi}_1$, then the functional
$G^{orilcz}_\varphi (\cdot)$ on $\cK_0$ is
continuous with respect to the Hausdorff distance. \et
 \bpf Let $\varphi\in \widehat{\Phi}_1$. The upper semicontinuity has been stated after Proposition \ref{p22--222}. To conclude the continuity, it is enough to prove the lower semicontinuity.

 To this end, we need the following statement: if $K_i\rightarrow K$ as $i\rightarrow \infty$ with $K_i, K\in \cK_0$ for all $i\geq 1$, there exists  a constant $R'>0$ such that $M_i\subset R'\ball
$ for all (given) $M_i\in T_\varphi K_i$, $i\geq 1$.  The proof basically
follows the idea in Lemma \ref{l3}. In fact, assume that there is no
constant $R'$ such that $M_i\subset R'\ball$ for $i\geq 1$. Let 
$R_i=\rho_{M_i}(u_i)=\text{max}\{\rho_{M_i}(u): u\in S^{n-1}\}.$
It follows from the Blaschke selection theorem and the compactness
of $\sphere$ that there is a subsequence of $\{K_i\}_{i\geq 1}$, which will not be relabeled, such that, $ R_i\rightarrow \infty$  and
$u_i\rightarrow v$ as $i\rightarrow \infty.$  For any given
$\varepsilon>0$, one has
\begin{eqnarray*}
V_\varphi(K,\ball)&=&
\lim_{i\rightarrow\infty}V_\varphi(K_i,\ball)\\ &\geq&
  \lim_{i\rightarrow\infty}\frac{1}{n}\int_{S^{n-1}}\varphi\left(\frac{h_{M_i}(u)}{h_{K_i}(u)}\right)h_{K_i}(u)dS_{K_i}(u)\\
&\geq&
\lim_{i\rightarrow\infty}\frac{1}{n}\int_{S^{n-1}}\varphi\left(\frac{\langle u,u_i\rangle_+}{({R_i}^{-1}+\varepsilon)\cdot R}\right)r dS_{K_i}(u)\\
&=&\frac{1}{n}\int_{S^{n-1}}  \varphi\left(\frac{\langle u,v\rangle_+}{\varepsilon\cdot R}\right)r dS_K(u)\\
&\geq&
\frac{1}{n}\int_{\Sigma_{j_0}(v)}\varphi\left(\frac{1}{\varepsilon\cdot  j_0\cdot R}\right)r dS_K(u)\\
&=& \varphi\left(\frac{1}{\varepsilon\cdot j_0\cdot R}\right)\cdot
\frac{r}{n}\cdot\frac{c_1}{2}
\end{eqnarray*}
where $r, R>0$ are constants such that $r\ball \subset K_i, K\subset
R\ball$ for all $i\geq 1$.  A contradiction (i.e.,
$V_\varphi(K,\ball)\geq \infty$) is obtained by taking
$\varepsilon\rightarrow 0^+$ and the fact that
$\lim_{t\rightarrow\infty} \varphi(t)=\infty$.

Now let us prove the lower semicontinuity of
$G^{orilcz}_\varphi(\cdot)$ and the continuity then follows. Let
$l=\liminf_{i\rightarrow \infty} G^{orlicz}_\varphi(K_i).$ There is a subsequence of $\{K_i\}_{i\geq 1}$, say
$\{K_{i_k}\}_{k\geq 1}$, such that, $l=\lim_{k\rightarrow \infty}
G^{orlicz}_\varphi(K_{i_k}).$ From the arguments in the previous paragraph, one sees
that $\{M_{i_k}\}_{k\geq 1}$ is uniformly bounded. The Blaschke
selection theorem and Lemma \ref{l1} imply that there exists a
subsequence of $\{M_{i_k}\}_{k\geq 1}$, which will not be relabeled,
and a convex body $M\in\cK_0$ such that $M_{i_k}\rightarrow M$ as
$k\rightarrow \infty$ and $|M^\circ|=\omega_n$. Proposition
\ref{p22--222} yields
$$G^{orlicz}_\varphi(K_{i_k})=nV_\varphi(K_{i_k}, M_{i_k})\rightarrow
nV_\varphi(K, M)\geq G^{orlicz}_\varphi(K) \ \ \ \ \mathrm{as} \ \ \
k\rightarrow \infty.$$   Hence, $\liminf_{i\rightarrow\infty}
G^{orlicz}_\varphi(K_i)\ge G^{orlicz}_\varphi(K)$ and this completes
the proof. \epf Similar to Proposition \ref{continuity of petty
bodies}, we can prove that if $\varphi\in
\widehat{\Phi}_1$ is convex, then $T_\varphi:
 \cK_0\mapsto \cK_0$ is continuous with respect to the Hausdorff distance.

\section{The Orlicz geominimal surface areas with respect to $\cK_e$ and the related Orlicz-Petty bodies}
\label{section-symmetric-geom}

In Sections \ref{section-homogeneous-petty} and
\ref{section-nonhomogeneous-gemo}, we prove the existence of the 
Orlicz-Petty bodies and the continuity for the Orlicz geominimal
surface areas under the condition $\phi\in \widehat{\Phi}_1$. For $\phi\in \widehat{\Phi}_2\cup \widehat{\Psi}$,
our method fails. In fact, when  $\phi\in \widehat{\Phi}_2$, we can
prove the following result. \bp\label{Polytope-infty-case} Let
$\phi, \varphi\in  \widehat{\Phi}_2$ and $K\in\cK_0$ be a polytope. Then  
$$\affineg(K)=0 \ \ \ \mathrm{and} \ \ \ \  G^{orlicz}_{\varphi}(K)=\infty.$$
  \ep
\begin{proof} Let $\phi\in  \widehat{\Phi}_2$ and $K\in\cK_0$ be a polytope.  Then the surface area measure of $K$ is concentrated on finite directions, say $\{u_1, \cdots, u_m\}\subset \sphere$. As $\affineg(K)$ is $SL(n)$ invariant, we can assume that, without loss of generality, $S_K(u_1)>0$ and $u_1=e_1$ with $\{e_1, \cdots, e_n\}$ the canonical orthonormal basis of $\bbR^n$.

Let $\epsilon>0$ and
$A_{\epsilon}=\mathrm{diag}(\epsilon, b_2, \cdots, b_n)$ with
constants $b_2, \cdots, b_n>0$  such that $b_2\cdots
b_n=1/\epsilon.$ Clearly $\det A_{\epsilon}=1$ and then
$A_{\epsilon}\in SL(n)$. Let $L_{\epsilon}=A_{\epsilon}K\in \cK_0$
and $\lambda_{\epsilon}=\widehat{V}_{\phi} (K, L_{\epsilon})$. Then,
$h_{L_{\epsilon}}(e_1)=\epsilon\cdot h_K(e_1)$ for all $\epsilon>0$
and \begin{eqnarray*} 1&=&\int_{S^{n-1}}\phi\left(\frac{n|K|\cdot
h_{L_{\epsilon}}(u)}{\lambda_{\epsilon} \cdot h_K(u)}\right)\dV \\
&=& \frac{1}{n|K|} \cdot  \sum_{i=1}^m  \phi\left(\frac{n|K|\cdot
h_{L_{\epsilon}}(u_i)}{\lambda_{\epsilon} \cdot
h_K(u_i)}\right)h_K(u_i) S_K(u_i) \\ &\geq & \frac{1}{n|K|}\cdot
\phi\left(\frac{n|K|\cdot h_{L_{\epsilon}}(e_1)}{\lambda_{\epsilon}
\cdot h_K(e_1)}\right)h_K(e_1) S_K(e_1) \\&=& \frac{1}{n|K|}\cdot
\phi\left(\frac{n|K|\cdot \epsilon}{\lambda_{\epsilon}
}\right)h_K(e_1) S_K(e_1). \end{eqnarray*} Assume that
$\inf_{\epsilon>0} \lambda_{\epsilon}> 0$. There exists a
constant $c>0$ such that $\lambda_{\epsilon}>c$ for all
$\epsilon>0$. The above inequality and the fact that $\phi\in
\widehat{\Phi}_2$ is decreasing imply
\begin{eqnarray*}
1 \geq  \frac{1}{n|K|}\cdot    \phi\left(\frac{n|K|\cdot
\epsilon}{\lambda_{\epsilon} }\right)h_K(e_1) S_K(e_1) \geq
\frac{1}{n|K|}\cdot    \phi\left(\frac{n|K|\cdot
\epsilon}{c}\right)h_K(e_1) S_K(e_1). \end{eqnarray*} Recall that
$\lim_{t\rightarrow 0} \phi(t)=\infty$ as $\phi\in
\widehat{\Phi}_2\subset \mathcal{D}$. A contradiction (i.e., $1\geq
\infty$) is obtained if we let $\epsilon\rightarrow 0^+$. This means
that
$$\inf_{\epsilon>0}\lambda_{\epsilon}=\inf_{\epsilon>0}\widehat{V}_{\phi}(K,
L_{\epsilon})=0.$$ On the other hand,
$\vrad(L_{\epsilon}^\circ)=\vrad(K^\circ)$ for all $\epsilon>0$.
This yields that \begin{eqnarray*} 0 \leq
\affineg(K)=\inf_{L\in\cK_0}\{\widehat{V}_\phi(K,\mathrm{vrad}(L^\circ)L)\}
\leq  \inf_{\epsilon>0}\{\widehat{V}_\phi(K,
\mathrm{vrad}(L_{\epsilon}^\circ)L_{\epsilon})\}=0.  \end{eqnarray*}

For the nonhomogeneous Orlicz $L_{\varphi}$ geominimal surface area,
the proof follows along the same lines. In fact, for all
$\epsilon>0$,
\begin{eqnarray*} V_{\varphi}(K,  \mathrm{vrad}(L_{\epsilon}^\circ)L_{\epsilon})&=&
\frac{1}{n} \int_{\sphere} \varphi\bigg(\frac{\vrad(K^\circ)
h_{L_{\epsilon}}(u)} {h_K(u)}\bigg)h_K(u)\,dS_K(u)\\ &\geq&
\frac{1}{n} \cdot \varphi (\vrad(K^\circ) \cdot \epsilon)\cdot h_K(e_1)\cdot S_K(e_1).
\end{eqnarray*} and  the desired result follows $$G^{orlicz}_{\varphi}(K)=
\sup_{L\in\cK_0}\{nV_\varphi(K, \mathrm{vrad}(L^\circ)L)\}  \geq 
\sup_{\epsilon>0} \{nV_{\varphi}(K,
\mathrm{vrad}(L_{\epsilon}^\circ)L_{\epsilon})\}=\infty.$$ This completes the proof. \end{proof}

An immediate consequence of Proposition \ref{Polytope-infty-case} is
that for $\phi\in \widehat{\Phi}_2$, the homogeneous Orlicz
$L_{\phi}$ geominimal surface area is not continuous but only upper
semicontinuous on $\cK_0$ with respect to the Hausdorff distance. To
this end, let $K=\ball$. One can find a sequence of polytopes
$\{P_i\}_{i\geq 1}$ such that $P_i\rightarrow \ball$ as
$i\rightarrow \infty$ with respect to the Hausdorff distance. However, one
cannot expect to have $\affineg(P_i)\rightarrow \affineg(\ball)$ as
$i\rightarrow \infty$, since $\affineg(P_i)=0$ for all $i\geq 1$ and
$\affineg(\ball)=n\omega_n>0$. Moreover, if $\phi\in \widehat{\Phi}_2$ and $K$ is a polytope, the Orlicz-Petty bodies for $K$ do not exist (i.e., $\widehat{T}_{\phi}K=\emptyset$). This is because $\affineg(K)=0$, but $\widehat{V}_{\phi}(K, M)>0$ for $M\in \widehat{T}_{\phi}K\subset \cK_0$ if $\widehat{T}_{\phi}K\neq \emptyset$.   Similarly,  the nonhomogeneous Orlicz
$L_{\varphi}$ geominimal surface area is not continuous but only lower
semicontinuous on $\cK_0$ with respect to the Hausdorff distance as
$G_{\varphi}^{orlicz}(P_i)=\infty$ for all $i\geq 1$. Moreover, if $\varphi\in \widehat{\Phi}_2$ and $K$ is a polytope, the Orlicz-Petty bodies for $K$ do not exist. 

Our method to prove the existence of the Orlicz-Petty
bodies in Sections \ref{section-homogeneous-petty} and
\ref{section-nonhomogeneous-gemo}
  heavily relies on the value of the Orlicz mixed volumes of $K$ and line segements $[0, v]=\{tv: t\in [0, 1]\}$  for $v\in \sphere$ (for instance $\widehat{V}_{\phi}(K, [0, v])$ in Section \ref{section-homogeneous-petty}).  However, $\widehat{V}_{\phi}(K, [0, v])$ are always $0$ for all $v\in \sphere$  if $\phi\in\mathcal{D}$.  It seems impossible to prove the existence of  the Orlicz-Petty bodies for $\phi\in \mathcal{D}$ and for general (even with enough smoothness) convex bodies $K\in \cK_0$.

When $\phi(t)= t^p$ for $p\in (-1, 0)$,  one can calculate
that, for all $v\in \sphere$  (see e.g., \cite{Zhang2007}),
\be\label{p-projection} \int_{S^{n-1}}|\langle u,v \rangle|^p
d\sigma(u)=C_{n, p},
 \ee
where $C_{n, p}>0$ is a finite constant depending on $n$ and $p$.
Note that the integrand includes $|\langle u, v\rangle|$ rather than
$\langle u, v\rangle_+$. This suggests that our method in Sections
\ref{section-homogeneous-petty} and
\ref{section-nonhomogeneous-gemo} may still work for smooth enough
$K\in\cK_0$ and a modified Orlicz geominimal surface
area.

Our modified Orlicz geominimal surface area is given by the
following definition. Recall that $\cK_e$ is the set of all origin-symmetric
convex bodies.

\bd Let $K\in \cK_0$ and $\phi\in \widehat{\Phi}$. The homogeneous
Orlicz $L_{\phi}$ geominimal surface area of $K$ with respect to
$\cK_e$ is defined by \be\label{symmetric-definition}
\widehat{G}_\phi^{orlicz}(K, \cK_e)=\inf\{\widehat{V}_\phi(K, L):\
L\in\cK_e \quad \mathrm{with}\quad |L^\circ|=\omega_n\}.
 \ee While if $\phi\in \widehat{\Psi}$, $\affineg(\cdot, \cK_e)$ can be defined similarly with `` $\inf$" replaced by `` $\sup$".
 \ed

Properties for $\widehat{G}_\phi^{orlicz}(\cdot, \cK_e)$, such as
affine invriance, homogeneity, affine isoperimetric inequalities
(requiring $K\in \cK_e$), and continuity if $\phi\in
\widehat{\Phi}_1$,  are the same as those for $\affineg(\cdot)$ proved in Sections
\ref{section-homogeneous-geom} and \ref{section-homogeneous-petty}.
The details are left for readers.

In the rest of this section, we will prove the existence of the
Orlicz-Petty bodies and the ``continuity" of $\affineg(\cdot, \cK_e)$
for certain $\phi\in \widehat{\Phi}_2$.  We will work on convex
bodies $K\in C_+^2$. A convex body $K$  is said to be in $C_+^2$ if
$K$ has $C^2$ boundary and positive curvature function $f_K$.
Hereafter, the curvature function of $K$ is the function
$f_K:\sphere\rightarrow (0, \infty)$ such that
$$f_K(u)=\frac{\,dS_K(u)}{\,d\sigma(u)} \ \ \ \mathrm{for} \ \
u\in \sphere. $$ Let $\phi\in \widehat{\Phi}_2$ be such that  for all $x\in
\bbR^n$,
\begin{equation}\label{condition on phi} \int_{\sphere}\phi(
|\langle u, x\rangle|)\,d\sigma(u) <\infty \ \ \ \mathrm{and} \ \ \
\lim_{\|x\|\rightarrow \infty} \int_{\sphere}\phi( |\langle u,
x\rangle|)\,d\sigma(u)=0.
\end{equation} Note that $\phi(t)=t^p$ for $p\in (-1, 0)$ satisfies
the condition (\ref{condition on phi}) due to formula
(\ref{p-projection}). Moreover, (\ref{condition on phi})  is equivalent to, for all $s>0$,
\begin{equation*} \int_{\sphere}\phi(
s\cdot |\langle u, e_1\rangle|)\,d\sigma(u) <\infty \ \ \ \mathrm{and} \ \ \
\lim_{s\rightarrow \infty} \int_{\sphere}\phi(s\cdot |\langle u,
e_1\rangle|)\,d\sigma(u)=0.
\end{equation*}

\bp \label{p3-07-22} Let $K\in C^2_+$ and  $\phi\in\widehat{\Phi}_2$
satisfy (\ref{condition on phi}). Then there exists $M
\in\cK_{e}$ such that
\begin{displaymath}
\widehat{G}^{orlicz}_\phi(K, \cK_e)=\widehat{V}_\phi(K, M) \quad
\mathrm{and}\ \quad |M^\circ|=\omega_n.
\end{displaymath}
 \ep

\begin{proof} Let $K\in C^2_+$. Its curvature function $f_K$ is continuous on $\sphere$ and hence has maximum which will be denoted by $F_K<\infty$.  By (\ref{symmetric-definition}), for $\phi\in \widehat{\Phi}_2$, there exists a sequence
$\{M_i\}_{i\geq 1}\subset \mathcal{K}_{e}$ such that
$\widehat{V}_\phi(K,M_i)  \rightarrow \affineg(K, \cK_e)$ as
$i\rightarrow\infty$, $2 \widehat{V}_\phi(K,\ball)\ge
\widehat{V}_\phi(K,M_i)$ and $|M_i^\circ|=\omega_n$  for all $i\geq
1$. Again let $R_i=\rho_{M_i}(u_i)=\max\{\rho_{M_i}(u):u\in
S^{n-1}\}.$ Then $h_{M_i}(u)\ge R_i\cdot|\langle u,u_i\rangle|$ for
all $u\in \sphere$ and all $i\geq 1$. Corollary
\ref{corollary:homogeneous-2}, together with (\ref{condition on phi})
and  the fact that $\phi\in \widehat{\Phi}_2$ is decreasing, implies
that, for all $i\geq 1$,
\begin{eqnarray*}
1&=& \int_{S^{n-1}}\phi\left(\frac{n|K|\cdot h_{M_i}(u)}
{\widehat{V}_\phi(K,M_i)\cdot h_K(u)}\right)\,d\widetilde{V}_K(u)\\
&\leq& \int_{S^{n-1}}\phi\left(\frac{n|K|\cdot R_i\cdot|\langle
u,u_i\rangle|}{2\widehat{V}_\phi(K,B_2^n)\cdot
h_K(u)}\right)\cdot \frac{h_K(u)f_K(u)}{n|K|}\,d\sigma(u)\\
&\leq&  \int_{S^{n-1}}\phi\left(\frac{n|K|\cdot R_i\cdot|\langle
u,u_i\rangle|}{2\widehat{V}_\phi(K,B_2^n)\cdot R_K}\right)\cdot
\frac{R_KF_K}{n|K|}\,d\sigma(u)<\infty.
\end{eqnarray*}
Assume that $\sup_{i\geq 1}R_i=\infty$. Without loss of generality,
let $\lim_{i\geq 1}R_i=\infty$ and $$x_i=\frac{n|K|\cdot R_i\cdot
u_i}{2 \widehat{V}_\phi(K,B_2^n)\cdot R_K}. $$ Then $\lim_{i\rightarrow \infty}\|x_i\|=\infty.$ It follows
from (\ref{condition on phi}) that
\begin{eqnarray*}
1&\leq& \frac{R_KF_K}{n|K|} \cdot \lim_{i\rightarrow \infty}
\int_{S^{n-1}}\phi\left(|\langle u,
x_i\rangle|\right)\,d\sigma(u)=0.
\end{eqnarray*} This is a contradiction and hence $\sup_{i\geq 1}R_i<\infty$. In other words, the sequence $\{M_i\}_{i\geq 1}$ is uniformly bounded. By the Blaschke selection
theorem, there exists a convergent subsequence of $\{M_i\}_{i\geq
1}$ (still denoted by $\{M_i\}_{i\geq 1}$) and a convex body
$M\in\cK$ such that $M_i\rightarrow M$ as $i\rightarrow \infty$. As
$|M_i^\circ|=\omega_n$ for all $i\geq 1$, Lemma \ref{l1} gives
$M\in\cK_e$ and $|M^\circ|=\omega_n$. Proposition \ref{p2} concludes
that $M$ is the desired body.
\end{proof}

\bd Let $K\in C^2_+$ and  $\phi\in\widehat{\Phi}_2$  satisfy
(\ref{condition on phi}). A convex body $M\in \cK_e$ is said to be an $L_{\phi}$ Orlicz-Petty
body of $K$ with respect to $\cK_e$, if $M\in\cK_e$ satisfies   \be
\nonumber \widehat{G}^{orlicz}_\phi(K,
\cK_e)=\widehat{V}_\phi(K, M)  \quad \mathrm{and} \quad
|M^\circ|=\omega_n. \ee Denote by $\widehat{T}_{\phi}(K, \cK_e)$ the
set of all such bodies. \ed

\bt \label{L-U-3} Let $\phi\in\widehat{\Phi}_2$  satisfy
(\ref{condition on phi}). Assume that $\{K_i\}_{i\geq 0} \subset
C^2_+$ such that $K_i\rightarrow K_0$ as $i\rightarrow \infty$ and
$\{f_{K_i}\}_{i\geq 1}$ is uniformly bounded on $\sphere$. Then
$$\lim_{i\rightarrow \infty} \affineg(K_i, \cK_e)=\affineg(K_0,
\cK_e).$$ \et

\begin{proof}
As $K_i\rightarrow K_0$,  
there exist $r, R>0$ such that $ r\ball\subset K_i\subset R\ball$ for all $i\geq 0.$  We claim that there is a
finite constant $R'>0$ such that $M_i\subset R'\ball$ for all (given)
$M_i\in \widehat{T}_\phi (K_i, \cK_e)$, $i\geq 1$. Suppose that
there is no such finite constant. Without loss of generality, assume
that $\lim_{i\rightarrow \infty} R_i=\infty$ and $u_i\rightarrow v$
(due to the compactness of $\sphere$) as $i\rightarrow \infty$,
where again
$$R_i=\rho_{M_i}(u_i)=\max\{\rho_{M_i}(u):u\in S^{n-1}\}.$$ As
before, $h_{M_i}(u)\ge R_i\cdot|\langle u,u_i\rangle|$ for all $u\in
\sphere$ and $i\geq 1$. Corollary \ref{corollary:homogeneous-2},
together with (\ref{condition on phi}) and  the fact that $\phi\in
\widehat{\Phi}_2$ is decreasing, implies that, for all $i\geq 1$,
\begin{eqnarray*}
1&=& \int_{S^{n-1}}\phi\left(\frac{n|K_i|\cdot h_{M_i}(u)}
{\widehat{V}_\phi(K_i, M_i)\cdot h_{K_i}(u)}\right)\,d\widetilde{V}_{K_i}(u)\\
&\leq& \int_{S^{n-1}}\phi\left(\frac{n|K_i|\cdot R_i\cdot|\langle
u,u_i\rangle|}{\widehat{V}_\phi(K_i ,B_2^n)\cdot
h_{K_i}(u)}\right)\cdot \frac{h_{K_i}(u)f_{K_i}(u)}{n|{K_i}|}\,d\sigma(u)\\
&\leq&  \int_{S^{n-1}}\phi\left(\frac{r^{n+1}\cdot R_i\cdot|\langle
u,u_i\rangle|}{R^{n+1}}\right)\cdot \frac{R\cdot F_0}{n\omega_n\cdot
r^n}\,d\sigma(u),
\end{eqnarray*} where the last inequality follows from Lemma \ref{l-bounded} and $F_0$ is the uniform bound of $\{f_{K_i}\}_{i\geq 1}$ on $\sphere$ (i.e., $F_0=\sup_{i\geq 1} \sup_{u\in \sphere} f_{K_i}(u)$). As in the proof of Proposition \ref{p3-07-22}, one gets
\begin{eqnarray*}
1\leq\lim_{i\rightarrow\infty}
\int_{S^{n-1}}\phi\left(\frac{r^{n+1}\cdot R_i\cdot|\langle
u,u_i\rangle|}{R^{n+1}}\right)\cdot \frac{R\cdot F_0}{n\omega_n\cdot
r^n}\,d\sigma(u)=0,
\end{eqnarray*} which is a contradiction. Hence there is a finite constant $R'>0$ such that $M_i\subset R'\ball$  for all (given) $M_i\in \widehat{T}_\phi (K_i, \cK_e)$, $i\geq 1$. In other words, $\{M_i\}_{i\geq 1}$ is uniformly bounded.

Let $l=\liminf_{i\rightarrow \infty} \affineg(K_i, \cK_e).$ Clearly, one
can find a subsequence $\{K_{i_k}\}_{k\geq 1}$ such that
$l=\lim_{k\rightarrow \infty} \affineg(K_{i_k}, \cK_e).$ By the Blaschke
selection theorem and Lemma \ref{l1}, there exists a subsequence of
$\{M_{i_k}\}_{k\geq 1}$ (still denoted by $\{M_{i_k}\}_{k\geq 1}$)
and a body $M\in\cK_e$, such that,  $M_{i_k}\rightarrow M$ as
$k\rightarrow \infty$ and $|M^\circ|=\omega_n$. Proposition \ref{p2}
then yields $$ \widehat{G}^{orlicz}_\phi(K_{i_k},
\cK_e)=\widehat{V}_\phi(K_{i_k}, M_{i_k})\rightarrow
\widehat{V}_\phi(K_0, M) \ \ \ \mathrm{as} \ \ \ k\rightarrow
\infty.$$ By (\ref{symmetric-definition}), one has  $$
\widehat{G}^{orlicz}_\phi(K_0, \cK_e)\leq \widehat{V}_\phi(K_0,
M)=\lim_{k\rightarrow\infty}
\widehat{G}^{orlicz}_\phi(K_{i_k}, \cK_e)=\liminf_{i\rightarrow \infty}
\widehat{G}^{orlicz}_\phi(K_i, \cK_e).
$$ On the other hand, for any given $\epsilon>0$, by
(\ref{symmetric-definition}) and  Proposition \ref{p2}, there exists
a convex body $L_{\epsilon}\in \cK_e$ such that
$|L_{\epsilon}^\circ|=\omega_n$  and \begin{eqnarray*} \affineg(K_0,
\cK_e)+\epsilon > \widehat{V}_{\phi}(K_0, L_{\epsilon}) =
\limsup_{i\rightarrow \infty} \widehat{V}_{\phi}(K_i,
L_{\epsilon})\geq  \limsup_{i\rightarrow\infty} \affineg(K_i,
\cK_e).
\end{eqnarray*}
By letting $\epsilon\rightarrow 0$, one gets $\affineg(K_0,
\cK_e)\geq \limsup_{i\rightarrow\infty} \affineg(K_i, \cK_e)$ and
the desired limit follows.  \end{proof}

 Let $K\in \cK_0$ and $\varphi\in \widehat{\Phi}_1\cup \widehat{\Psi}$. The  nonhomogeneous Orlicz $L_{\varphi}$ geominimal surface area of $K$ with respect to $\cK_e$ can be defined by
\be\nonumber {G}_\varphi^{orlicz}(K, \cK_e)=\inf\{n {V}_\varphi(K, L):\
L\in\cK_e \quad \text{with}\quad |L^\circ|=\omega_n\}.
 \ee While if $\varphi\in \widehat{\Phi}_2$, ${G}_\varphi^{orlicz}(\cdot, \cK_e)$ can be defined similarly with `` $\inf$" replaced by `` $\sup$". Analogous results  to  Proposition \ref{p3-07-22} and Theorem \ref{L-U-3} can be proved for ${G}_\varphi^{orlicz}(\cdot, \cK_e)$ if $\varphi\in\widehat{\Phi}_2$  satisfies (\ref{condition on phi}). We leave the details for readers.

\vskip 2mm \noindent {\bf Acknowledgments.} The first author is
supported by AARMS, NSERC, NSFC (No. 11501185) and the Doctor
Starting Foundation of Hubei University for Nationalities (No.
MY2014B001). The third author is supported by a NSERC
grant.

\vskip 2mm \noindent Baocheng Zhu, \ \ \ {\small \tt zhubaocheng814@163.com}\\
{ \em 1. Department of Mathematics,  
 Hubei University for Nationalities,  
 Enshi, Hubei, China 445000}\\
{  \em 2.\  Department of Mathematics and Statistics,   Memorial University of Newfoundland,
   St.\ John's, Newfoundland, Canada A1C 5S7 }

\vskip 2mm \noindent Han Hong, \ \ \ {\small \tt honghan0917@126.com}\\
{ \em Department of Mathematics and Statistics,   Memorial University of Newfoundland,
   St.\ John's, Newfoundland, Canada A1C 5S7 }

\vskip 2mm \noindent Deping Ye, \ \ \ {\small \tt deping.ye@mun.ca}\\
{ \em Department of Mathematics and Statistics,
   Memorial University of Newfoundland,
   St.\ John's, Newfoundland, Canada A1C 5S7 }

\end{document}